\documentclass[12pt,a4paper]{article}
\title{Variational formulation of the Fokker-Planck equation with decay: a particle approach}

\author{Mark Peletier and Michiel Renger\footnote{Department of Mathematics and Computer Sciences and Institute for Complex Molecular Systems, Technische Universiteit Eindhoven. The research of the first author has benefited from support by the Leverhulme Trust and 
by the Initial Training Network ``FIRST'' of the Seventh Framework Programme of the European Community (grant agreement number 238702).} and Marco Veneroni\footnote{Department of Mathematics, University of Pavia.}}

\date{\today}

\usepackage{amsmath}
\usepackage{amssymb, amsthm}
\usepackage{color}
\usepackage{a4wide}
\usepackage{tikz}
\usetikzlibrary{decorations.pathreplacing}
\usepackage{url}

\newcounter{theorem}

\DeclareMathOperator{\Law}{Law}
\renewcommand{\ae}{\textrm{ a.e.}}

\let\ds\displaystyle

\def\Prob{\mathrm{Prob}}

\def\weakto{\rightharpoonup}
\def\longrightharpoonup{\relbar\joinrel\rightharpoonup}
\def\longweakto{\longrightharpoonup}
\def\hq{\hskip 0.5em}

\def\R{\mathbb R}
\def\N{\mathbb N}
\def\P{\mathcal P}
\def\E{\mathcal E}
\def\M{\mathcal M}

\def\F{\mathcal F}
\def\K{\mathcal K}
\def\J{\mathcal J}
\def\S{\mathcal S}
\def\E{\mathcal E}

\def\NN{_{N\!N}}
\def\ND{_{N\!D}}

\def\DD{_{D\!D}}
\def\NT{_{N\!T}}

\def\Diff{_\text{\textit{Df}}}
\def\FP{_\text{\textit{FP}}}
\def\Dec{_\text{\textit{Dc}}}
\def\DDec{_\text{\textit{DfDc}}}
\def\FPDec{_\text{\textit{FPDc}}}
\def\Kcontracted{\overline K^h\DDec}
\def\init{\overline}
\def\lapl{\Delta}
\def\grad{\nabla}
\renewcommand{\div}{\mathop{\mathrm{div}}}
\DeclareMathOperator*\argmin{arg\,min}
\DeclareMathOperator\Grad{grad}

\newtheorem{defn}[theorem]{Definition}
\newtheorem{theo}[theorem]{Theorem}
\newtheorem{coro}[theorem]{Corollary}
\newtheorem{lem}[theorem]{Lemma}

\newtheorem{conj}[theorem]{Conjecture}
\newenvironment{rem}%
  {\par\medbreak\refstepcounter{theorem}%
    \noindent\textbf{Remark~\thetheorem. }}%
  {\qed\par\medskip}

\begin{document}
\maketitle
\begin{abstract}We introduce a stochastic particle system that corresponds to the Fokker-Planck equation with decay in the many-particles limit, and study its large deviations. We show that the large-deviation rate functional corresponds to an energy-dissipation functional in a Mosco-convergence sense. Moreover, we prove that the resulting functional, which involves entropic terms and the Wasserstein metric, is again a variational formulation for the Fokker-Planck equation with decay.
\end{abstract}


\newpage
 
 \section{Introduction}

\subsection{On the origin of Wasserstein gradient flows}

Since the introduction of the Wasserstein gradient flows in 1997--8~\cite{JordanKinderlehrerOtto97,JKO1998, Otto1998, Otto2001} it has become clear that a very large number of well-known parabolic partial differential equations and other evolutionary systems can be written as gradient flows. Examples of these are non-linear drift-diffusion equations~\cite{Agueh2005}, diffusion-drift equations with non-local interactions~\cite{CarrilloMcCannVillani2003}, higher-order parabolic equations~\cite{Otto1998a,GiacomelliOtto2001,Glasner2003,MatthesMcCannSavare2009,GianazzaSavareToscani2009}, moving-boundary problems~\cite{Otto1998a,PortegiesPeletier2010}, and chemical reactions~\cite{Mielke11}. The parallel development of rate-independent systems introduced similar variational structures for friction~\cite{EfendievMielke06}, delamination~\cite{KocvaraMielkeRoubicek06}, plasticity~\cite{Mielke04}, phase transformations~\cite{MielkeTheilLevitas02},  hysteresis~\cite{MielkeTheil04}, and various other  phenomena. Further generalisations are suggested by taking limits of gradient flows, as in the case of Kramers' equation for chemical reactions~\cite{Arnrich2012}.

This multitude of gradient-flow structures does raise questions. Before 1997, for instance, it was widely believed that convection-diffusion equations could not be gradient flows. This belief was contradicted by~\cite{JordanKinderlehrerOtto97,JKO1998}; apparently the question `which systems can be gradient flows' is a non-trivial one. As another example, common building blocks of these gradient-flow structures, such as the Wasserstein metric, appear to be mathematical, non-physical constructs---can one give these an interpretation in terms of physics, chemistry, or other modelling contexts?

In~\cite{Adams2011} the authors give a suggestion for an organising principle behind the observed variety in systems and gradient flows. For the example of the entropy-Wasserstein gradient flow (see below) they show how the gradient-flow structure itself is closely related to the probabilistic structure of a system of stochastic particles. This connection explains many aspects of the gradient flow, such as the origin of both the entropy and the Wasserstein metric and the interpretation of the discrete-time approximation. 

The result of \cite{Adams2011} also suggests that this connection between gradient-flow structures and stochastic particle systems may be much more general. 
In this paper we explore this idea for the following diffusion equation with convection and decay:
\begin{equation}
\label{eq: diffusion drift decay}
	\partial_t u = \lapl u + \div \!\left( u \, \grad\Psi \right) - \lambda u, 	\qquad \text{in }\R^d \times (0,\infty), \\
\end{equation}
with $\Psi\in C_b^2(\R^d)$ and $\lambda\geq 0$. We contribute two main results to the theory of this type of equations: first, we derive a new gradient-flow  formulation for equation~\eqref{eq: diffusion drift decay}, and secondly, since this formulation is constructed along the lines of \cite{Adams2011}, we automatically connect this gradient flow to microscopic systems of diffusing particles, and show that the gradient-flow structure arises from the probabilistic structure of these particle systems. 

%
%

The paper is organised as follows. In the remainder of this introductory section we develop the required concepts and formulate the main aim of this paper in a little more detail. Next, we recall the central notions of this paper in Section~\ref{sec: background}. We proceed with our microscopic models and the corresponding results in Sections~\ref{sec: FP} and \ref{sec: diff dec}, and we wrap up with a general discussion in Section~\ref{sec: discussion}. In the Appendix we give a description and the proof of an existing large-deviation result in a language that is more suited to this paper.

\subsection{Variational formulations} 
\label{sec: Wasserstein gradient flows}

In this paper we study iterative variational schemes on some space $\mathcal X$ of the form
\begin{equation}
\label{eq:variational scheme}
 \text{Given $\rho^{k-1}$, choose }\quad \rho^k\in \argmin_{\rho\in \mathcal X}\; \K^h(\rho|\rho^{k-1}),
\end{equation}
which will approximate the solution of an evolution equation as $h \to 0$. The following examples illustrate the main ideas.

\paragraph{Example 1: Hilbert-space gradient flows.} If $\mathcal X$ is a Hilbert space and the functional $\K^h$ is of the form
\begin{equation}
\label{eq:Hilbert minimisation}
	\K^h(\rho|\init\rho) = \F(\rho) + \frac1{2h} \|\rho-\init\rho\|^2
\end{equation}
for some smooth functional $\F$, then the minimisation problem~\eqref{eq:variational scheme} gives the stationarity condition
\[
\frac{\rho^k -\rho^{k-1}}h = -\Grad \F(\rho^k).
\]
In this one recognises the backward Euler approximation of the continuous-time gradient flow:
\begin{equation}
\label{eq:H-GF}
\partial_t \rho = -\Grad \F(\rho).
\end{equation}

The  time-discrete variational form \eqref{eq:Hilbert minimisation} illustrates how in gradient flows the evolution is driven by a trade-off between two competing effects. An \emph{energy functional} $\F:\mathcal X\to\R\cup\{\infty\}$ drives the system towards lower values of the energy; at the same time a \emph{dissipation mechanism} (here quantified by the norm $\|\cdot\|$) acts as a selection principle among all directions that decrease $\F$.
\medskip

If one chooses $\mathcal X=L^2(\R^d)$ and $\F(\rho)=\frac12\int\!\left|\grad \rho\right|^2$, then \eqref{eq:H-GF} simply becomes the diffusion equation. However, it is not possible to describe convection in this way. The next example shows that convection-diffusion equations are nevertheless gradient flows, in a more general context.

\paragraph{Example 2: Wasserstein gradient flows.} Instead of a Hilbert space, we now consider the metric space $\mathcal X=\P_2(\R^d)$ of probability measures with finite second moment, equipped with the Wasserstein metric $d$ (see Section~\ref{sec: wasserstein distance}). Similarly to \eqref{eq:Hilbert minimisation}, let (where the subscript $\textit{FP}$ stands for `Fokker-Planck'):
\begin{equation}
\label{def: KFP}
	\K^h\FP(\rho|\init\rho) := \tfrac12 \F(\rho) - \tfrac12 \F(\init\rho) + \frac{1}{4h} d^2\!\left(\init\rho,\rho\right),
\end{equation}
where $\mathcal{F}(\rho)=\S(\rho) + \E(\rho)$ is the Helmholtz free energy, and 
\begin{align}
\label{def: entropy}
&\S(\rho) := \begin{cases}
     \ds\int \! \log f(y) \, \rho(dy),     & \text{if }\rho(dy) = f(y) dy,\\
     \;\infty,                                           & \text{otherwise},
   \end{cases}\\
&\E(\rho) := \int \! \Psi (y)\, \rho(dy).\notag
\end{align}
are the (negative) Gibbs-Boltzmann entropy and the energy arising from a potential $\Psi$. Note that in comparison to~\eqref{eq:Hilbert minimisation} we have subtracted the free energy of the previous state, and multiplied the expression by $1/2$. Both are done in view of the connection to large-deviation rate functionals that we establish below; of course neither change affects the minimisation properties of $\K^h\FP(\,\cdot\,;\init\rho)$.

It was first observed by Jordan, Kinderlehrer and Otto~\cite{JordanKinderlehrerOtto97,JKO1998} that the time-discrete process defined by~\eqref{eq:variational scheme} and~\eqref{def: KFP} converges to the solution of the Fokker-Planck equation:
\begin{equation}
\label{eq: Fokker-Planck}
	\partial_t u = \lapl u + \div\!\left( u \, \grad\Psi\right), \qquad \text{in }\R^d \times (0,\infty).
\end{equation}
We see that, in the same sense as the previous example, the Fokker-Planck equation is a gradient flow of free energy with respect to the Wasserstein metric. For future reference, we duplicate their main theorem here (where the superscript $a$ denotes absolutely continuous):
\begin{theo}[\cite{JKO1998}]
\label{theo: JKO} Let  $\rho^0 \in \P_2^a(\R^d)$, and define the sequence $\{\rho^{h,k}\}_{k \geq 0}$ by:
	\begin{align*}
		&\rho^{h,0} = \rho^0, \\
		&\rho^{h,k} \in \underset{\rho \in \P_2(\R^d)}{\arg \min} \,  	\K^h\FP(\rho|\rho^{h,k-1}), & k\geq 1.
	\end{align*}
	These minimisers exist uniquely, and as $h\to 0$, the function $\rho^{h,\lfloor t/h \rfloor}$ converges weakly in $L^1(\R^d\times(0,T))$ to the solution of \eqref{eq: Fokker-Planck} with initial condition $\rho^0$.
\end{theo}
\noindent Actually, \cite{JKO1998} provides an argument to extend this result to weak convergence in $L^1(\R^d)$ for almost every $t\in(0,T)$ and strong convergence in $L^1(\R^d,(0,T))$.

While various generalisations of Hilbert-space gradient flows were known for some time~\cite{AlmgrenTaylorWang93,DeGiorgi2006,LuckhausSturzenhecker95}, this result meant a breakthrough by extending the concept to a large and important class of evolution equations. In addition to inspiring a great amount of research into gradient flows in Wasserstein spaces and in general metric spaces, in a variety of functional-analytic settings~\cite{Otto2001,Mielke05a,Ambrosio2008,RossiMielkeSavare2008}, it also gave rise to many fruitful connections between partial differential equations, optimal transport theory, geometry, functional inequalities, and probability; see \cite{Villani2003,Villani2009} for an overview.

\medskip

\textbf{Example 3: exponential decay.} 
As in some other cases~\cite{AlmgrenTaylorWang93,LuckhausSturzenhecker95}, it will be useful to consider more general time-discrete constructions, namely  of the form
\begin{equation}
\label{def:example decay}
  \K^h(a|\init a) \;= \F(a;\init a) + f^h(\init a,a),
\end{equation}
for some function $f^h$. In this example, fix some $0<r^h<1$ and let the state space be $\mathcal X=\R^+$. Take for $\F$ a mixing entropy with parameter $\init a$,
\begin{equation}
\label{def:example decay energy}
  \F(a;\init a) := a\log a +(\init a-a)\log (\init a- a),\qquad\text{for }0<a<\init a,
\end{equation}
and for $f^h$ the expression\footnote{As suggested by one of the referees, this particular form \eqref{def:example decay}$+$\eqref{def:example decay energy} arises as the quenched large-deviation rate of a system of independent exponentially distributed decay processes, with $a=\frac1n\#\{\text{non-decayed }X_i(h)\}$ and $\bar a=\frac1n\#\{\text{non-decayed } X_i(0)\}$.}
\begin{equation}
\label{def:example decay diss}
  f^h(\init a,a) := -a\log r^h - (\init a - a)\log (1-r^h).
\end{equation}
Then, the unique minimiser of~\eqref{def:example decay} is $a=r^h\;\init a$. While this construction may appear to be a convoluted way of arriving at this result, in fact it appears \emph{naturally} in the context of a specific stochastic system of particles, as we show below. In the limit $h\to0$ it will describe the term $-\lambda u$ in~\eqref{eq: diffusion drift decay} which is associated with decay, as is illustrated by the following simple result: 
\begin{theo}
\label{theo:simple-x}
Let $\K^h$ be given as in~(\ref{def:example decay}--\ref{def:example decay diss}) with $r^h:=e^{-\lambda h}$. Let $a^0\in \R^+$ be fixed and define the sequence $\{a^{h,k}\}_{k\geq0}$ by 
\begin{align*}
&a^{h,0} = a^0,\\
&a^{h,k} \in \argmin_{a \in \R^+}\; \K^h(a|a^{h,k-1}), \qquad k\geq1.
\end{align*}
Then as $h\to0$ the function $t\mapsto a^{h,\lfloor t/h\rfloor}$ converges in time to the solution $t\mapsto a^0 e^{-\lambda t}$ of $\partial_t u = -\lambda u$.
\end{theo}
\noindent
The proof follows from remarking that $a^{h,k} = a^0 e^{-\lambda kh}$.


Below we will consider this construction in integrated form:
\begin{equation*}
	\K^h\Dec(\rho|\init \rho) \;:=\; -\S(\init \rho) + \S(\rho) + \S(\init\rho\!-\!\rho) - \big|\rho\big|\log r^h - \big|\init\rho\!-\!\rho\big|\log (1-r^h)
\end{equation*}
(the subscript \textit{Dc} stands for `Decay equation') on the space of non-negative Borel measures $\M^+(\R^d)$ with the total variation norm $\big|\rho\big| := \rho(\R^d)$. Observe that compared to (\ref{def:example decay}--\ref{def:example decay diss}), we have an additional term $-\S(\init\rho)$. This term does not influence the minimiser, but we have added it here to ensure that the minimum is $0$, which will be needed below.

\medskip

\textbf{Synthesis of examples 2 and 3.} In the results that we prove in this paper, the last two examples are merged in a single variational scheme. In the simplest case, for instance, where $\Psi\equiv0$, the discrete algorithm approximating~\eqref{eq: diffusion drift decay} becomes
\begin{multline}
\label{min:intro-rhok-DD}
\rho^k \in\argmin_{\rho\in\M^+(\R^d)} \ \inf_{\rho\ND : |\rho+\rho\ND|=|\rho^{k-1}|} 
-\tfrac12 \S(\rho+\rho\ND) -\tfrac12 \S(\rho^{k-1}) 
  + \tfrac{1}{4h} d^2(\rho +\rho\ND,\rho^{k-1})\\ 
+ \S(\rho) + \S(\rho\ND) -|\rho|\log r^h - |\rho\ND|\log (1-r^h).
\end{multline}

To interpret the formula above, one should realise that the infimum over the measure $\rho\ND$ in the formula above  represents a choice: in each time step, the system designates a portion $\rho\ND\geq 0$ for decay (the index $N\!D$ stands for `Normal to Decayed'), while the other part $\rho\geq 0$ remains `normal'. 

The terms inside the infimum can be written as $\K^h\FP(\rho+\rho\ND|\rho^{k-1})+ \K^h\Dec(\rho|\rho+\rho\ND) $, and one can understand the structure of~\eqref{min:intro-rhok-DD} through this splitting. The functional $\K^h\FP(\rho+\rho\ND|\rho^{k-1})$ characterises a single time-step of diffusion of $\rho^{k-1}$, according to Theorem~\ref{theo: JKO}. Decay is left out of this step, since the joint mass $\rho+\rho\ND$ is independent of the distribution over normal ($\rho$) and decayed matter ($\rho\ND$). In a second step, given a choice for $\rho+\rho\ND$, the second functional $\K^h\Dec(\rho|\rho+\rho\ND)$  describes how the total mass $\rho+\rho\ND$ is divided over $\rho$ and $\rho\ND$, according to Theorem~\ref{theo:simple-x}.  As such, we can interpret $\rho+\rho\ND$ as an intermediate state between $\rho^{k-1}$ and $\rho$.



\subsection{From microscopic model to large deviations}

We claimed above that the approximation scheme arises naturally in the context of stochastic particle systems. We now describe this context. It is well known (going back at least to Einstein~\cite{Einstein1905}) that the diffusion equation 
\begin{equation}
\label{eq: diffusion equation}
	\partial_t u = \lapl u, \qquad \text{in }\R^d \times (0,\infty), \\	
\end{equation}
is the macroscopic (hydrodynamic, continuum) limit of a wide range of stochastic particle systems \cite{Masi1991}. Here we focus on one such system, composed of independent Brownian particles. 

More specifically, let all particles $1,\hdots,n$ be initially distributed according to some fixed $\init\rho \in \P(\R^d)$, and, for a fixed time interval $h>0$, let each particle $i=1,\hdots,n$ move to a new position $Y^h_i$, where the probability of moving from $x$ to $y$ is given by the density (which is identical for all particles)
\begin{equation}
\label{def: diffusion kernel} 
	\theta^h(y-x) := \frac{1}{(4\pi h)^{d/2}} \exp\Bigl{(}-\frac{|x-y|^2}{4h}\Bigr{)}.
\end{equation}

The empirical measure $L_n^h :=n^{-1}\sum_{i=1}^n \delta_{Y^h_i}$  then is a random probability measure that describes the distribution of all $n$ particles in space at time $h$. This measure converges (as $n\to \infty$) to $\init\rho\ast \theta^h$, the solution of \eqref{eq: diffusion equation} at time $h$ with initial condition $\init\rho$.


The speed of this convergence is characterised by a \emph{large-deviation principle}, which we discuss in Section~\ref{sec: large deviations}. It states that the probability of finding $L_n^h$ close to some $\rho \in \P(\R^d)$ converges exponentially to zero with rate $n\J^h\Diff(\rho|\init\rho)$ (the subscript stands for `Diffusion equation'):
\begin{equation*}
	\Prob(L_n^h\approx \rho|L_n^0\approx \init\rho) \sim \exp \bigl{(}-n \J^h\Diff(\rho|\init\rho)\bigr{)} \quad \text{  as } n\to \infty.
\end{equation*}
The \emph{rate functional} $\J^h\Diff(\,\cdot\,|\init\rho)$ is non-negative and minimised by the solution of \eqref{eq: diffusion equation} at time~$h$. 

\subsection{From large deviations to Wasserstein gradient flow}
\label{ssec:ldtowasser}
When restricting ourselves to the diffusion equation \eqref{eq: diffusion equation}, the gradient-flow functional \eqref{def: KFP} reduces to
\begin{equation*}
	\K^h\Diff(\rho|\init\rho) := \tfrac12 \S(\rho) - \tfrac12 \S(\init\rho) + \frac{1}{4h} d^2(\init\rho,\rho).
\end{equation*}
Recent results \cite{Adams2011,Duong2012} have shown that, under suitable assumptions, not only the minimisers of $\J^h\Diff$ and $\K^h\Diff$ have the same limit, but the two are in fact strongly related. Since we expect this statement to be generally true, we pose it here as a conjecture. It will be convenient to introduce the set:
\begin{equation*}
  \P_2^{\S}(\R^d):=\left\{\rho\in\P(\R^d):\int\!|x|^2\,d\rho<\infty, \S(\rho)<\infty\right\}.
\end{equation*}
\begin{conj}
\label{conj} For any fixed $\init\rho\in\P_2^\S(\R^d)$ there holds
	\begin{equation}
	\label{eq: diffusion gamma convergence}
		\J^h\Diff(\,\cdot\,|\init\rho) - \frac{1}{4h}d^2(\init\rho,\,\cdot\,) \xrightarrow[h\to 0]{M}\ \tfrac12 \S(\cdot) - \tfrac12 \S(\init\rho)=\K^h\Diff(\,\cdot\,|\init\rho) - \frac{1}{4h}d^2(\init\rho,\,\cdot\,)
	\end{equation}
	in the sense of Mosco convergence, where the lower bound holds in $\P_2(\R^d)$ with the narrow topology, and the recovery sequence holds in the topology defined by convergence in Wasserstein distance plus convergence in entropy $\S$	(see Section~\ref{sec:Mosco convergence}).
\end{conj}
This conjecture was first proven in \cite{Adams2011} under the restriction that both $\rho$ and $\init\rho$ in $\J^h\Diff(\rho|\init\rho)$ are sufficiently close to uniform distributions on a bounded interval in $\R$. In \cite{Duong2012}, the result was generalised to $\R$ for any $\init\rho$ with bounded Fisher information.

Note that the term $-(4h)^{-1} d^2(\init\rho, \,\cdot\,)$ appears on both sides of~\eqref{eq: diffusion gamma convergence}. The role of this term is to compensate the singular behaviour of both $\J^h\Diff$ and $\K^h\Diff$ in the limit $h\to0$. Morally, the conjecture states that
\begin{equation*}
  \text{as }h\to0, \qquad \J^h\Diff(\,\cdot\,|\init\rho) \approx \K^h\Diff(\,\cdot\,|\init\rho).
\end{equation*}
This connection shows how the functional $\K^h\Diff$, which defines the time-discretised gradient flow, can be interpreted physically: as the large-deviation rate functional of the microscopic model.

\subsection{Overview of this work}

In this article we extend the results of~\cite{Adams2011,Duong2012} to equation~\eqref{eq: diffusion drift decay}. Although the results in the latter already includes the Fokker-Planck equation~\eqref{eq: Fokker-Planck}, this paper uses very different techniques and yields results under different assumptions on the potential $\Psi$. The main results of this paper are of the same form as Theorem~\ref{theo: JKO} and Conjecture~\ref{conj}.

We divide the arguments, and the paper, into two parts. In the first part we discuss diffusion with drift but without decay ($\Psi\not\equiv 0$, $\lambda=0$ in~\eqref{eq: diffusion drift decay}). First we construct a system of Brownian particles with drift that models the Fokker-Planck equation \eqref{eq: Fokker-Planck}, and then derive a corresponding large-deviation principle. In our first main result, Theorem~\ref{theo: Fokker-Planck Gamma convergence}, we show that for small times the large-deviation rate functional of the micro model relates to $\K^h\FP$ in the same sense as in the Conjecture~\ref{conj} for the diffusion equation. Note that the expression for the gradient-flow functional $\K^h\FP$ is already known from~\cite{JKO1998}; the novelty of the current result lies in the connection to the microscopic particle system.

The second part of the paper concerns the diffusion equation with decay ($\lambda>0$, and for ease of notation we first take $\Psi\equiv 0$):
\begin{equation}
\label{eq: diffusion decay}
	\partial_t u = \lapl u - \lambda u,		 	\qquad \text{in }\R^d \times (0,\infty).
\end{equation}
Again, we devise a particle system that models this equation microscopically, and derive a corresponding large-deviation principle. In the second main result of this paper, Theorem~\ref{theo: diff dec gamma convergence}, we show that the large-deviation rate functional relates to an energy-dissipation functional~\eqref{def:KFPDec} in the same way as in Conjecture~\ref{conj}. Finally, in Theorem~\ref{theo: diff dec gradient flow} we show that the minimisers of this new functional indeed approximate the solution of \eqref{eq: diffusion decay} in the sense of Theorem~\ref{theo: JKO}. In this case, the novelty lies in both the expression of the energy-dissipation functional, and in its connection to the microscopic system.

\section{Background}
\label{sec: background}

\subsection{Wasserstein distance}
\label{sec: wasserstein distance}
In the Kantorovich formulation of the optimal transport problem, a transport plan between two measures $\init\rho, \rho \in \P(\R^d)$ is a measure in the set
\begin{equation*}
	\Gamma(\init\rho,\rho) := \bigl\{q \in \P(\R^d\times \R^d) : \pi^1 q=\init\rho \text{ and } \pi^2 q=\rho \bigr\},
\end{equation*}
where we denote the the marginals of $q$ by
\begin{align*}
  \pi^1 q(B) := q(B\times \R^d) && \text{ and }&& \pi^2q(B) := q(\R^d\times B) && \text{ for all Borel sets } B\subset \R^d.
\end{align*}
In the particular case of the 2-Wasserstein distance (henceforth simply called the Wasserstein distance), the unit cost of transporting an infinitesimal mass from position $x$ to $y$ is taken to be $|x-y|^2$. One can then ask for the optimal transport plan that transports all mass from a measure $\init\rho$ to another measure $\rho$. The minimum cost defines a metric on the space $\P_2(\R^d):=\{\rho\in\P(\R^d):\int\!|x|^2\,d\rho<\infty\}$ and is called the 
\begin{defn}[Wasserstein distance]
\begin{equation*}
	d^2(\init\rho,\rho) := \inf_{q\in\Gamma(\init\rho,\rho)} \iint |x-y|^2 \, q(dx\,dy).
\end{equation*}
\end{defn}
An important property of the Wasserstein distrance is that a sequence $\{\rho_h\}_h \in \P_2(\R^d)$ converges to $\rho$ in the Wasserstein distance as $h\to0$ if and only if \cite[Th.~7.12]{Villani2003}
\begin{enumerate}
\item $\rho_h \weakto \rho$ (see Section~\ref{sec:Mosco convergence}),
\item $\int\!x^2\,\rho_h(dx)\to\int\!x^2\,\rho(dx)$.
\end{enumerate}

Observe that the Wasserstein distance is still meaningful for measures $\init\rho,\rho\in\M^+(\R^d)$ that are not necessarily probability measures, as long as $|\init\rho|=|\rho|$. With this generalisation we have that
\begin{equation}
\label{ineq:Wasserstein-sum}
d^2(\rho_1+\rho_2,\rho_3+\rho_4) \leq d^2(\rho_1,\rho_3) + d^2(\rho_2,\rho_4) 
\qquad\text{for all }\rho_{1,2,3,4}\text{ with }|\rho_1| = |\rho_3| \text{ and }|\rho_2| = |\rho_4|.
\end{equation}
This property will be used later in the article.

\subsection{Large deviations}
\label{sec: large deviations}
Recall from the law of large numbers that with probability $1$, in the large-$n$ limit the expectation $\mathbb{E} L_1^h$ is the only event that occurs (see for example \cite[Th. 11.4.1]{Dudley1989}). In this limit, any other event is  considered a large deviation from this expected behaviour. A large-deviation principle characterises the unlikeliness of such event by the speed of convergence of its probability to $0$. To illustrate this, we briefly switch to a more abstract notation.

\begin{defn} A sequence~$X_n$ of random variables with variables in a topological space $\mathcal X$ satisfies the \emph{large-deviation principle} with speed $n$ and \emph{rate functional} $\J:\mathcal X\to [0,\infty]$ whenever:
\begin{enumerate}
	\item $\J$ is not identically $\infty$, and  $\J^{-1}\!\left[0,c\right]$ is compact for all $c < \infty$; 
	\item $\liminf_{n \to \infty} \tfrac{1}{n} \log \Prob(X_n\in U)\geq -\inf_{x \in U} \J(x)$ for all open sets $U \subset \mathcal X$; 
	\item $\limsup_{n \to \infty} \tfrac{1}{n} \log \Prob(X_n\in C) \leq -\inf_{\rho \in C} \J(x)$ for all closed sets $C \subset \mathcal X$.
\end{enumerate}
\end{defn}

The rate functional $\J$ is non-negative and achieves its minimum of zero at the most probable behaviour of $X_n$. The right-hand infimum reflects the general principle that ``any large deviation is done in the least unlikely of all the unlikely ways'' \cite[p.~10]{Hollander2000}. A related mathematical result is the \emph{contraction principle} \cite[Th.~4.2.1]{Dembo1998}, which states the following. Let $p:\mathcal{X}\to \mathcal{Y}$ be a continuous map, and $Y_n := p(X_n)$ the corresponding random variables. Then $Y_n$ satisfies a large-deviation principle similar to the one above, with rate functional $\inf_{x\in \mathcal X: p(x) = y} \J(x)$. This contraction principle will be used throughout this paper. For instance, it explains the role of the minimisation in~\eqref{min:intro-rhok-DD}.

\subsection{Mosco convergence}
\label{sec:Mosco convergence}

A useful tool in the study of sequences of minimisation problems is $\Gamma$-convergence~\cite{DalMaso93}. In particular, it is often used in the study of large deviations \cite[Lem.~2]{Adams2011} and gradient flows (cf. \cite{DeGiorgi2006,Sandier2004}). Moreover, in \cite{Leonard2007}, Gamma-convergence is used to connect large deviations to optimal transport. In many cases, it is convenient to require that the recovery sequence of the $\Gamma$-convergence exists in a stronger topology (cf. \cite[Rem.~2.0.5]{Ambrosio2008} or \cite{Mielke2012}): the resulting notion of convergence is known as Mosco-convergence~\cite{Mosco1969}. In results that are related to this paper, a further analysis reveals that Mosco-convergence is indeed satisfied (cf. \cite[Th.~3]{Adams2011},\cite[Th.~1.1]{Duong2012}). In this sense it provides a natural notion for the purpose of this study.

\begin{defn}Let $\mathcal{X}$ be a space with two first-countable (e.g. metrisable) topologies $\tau_w\subset \tau_s$. A sequence of functionals $\{\F^h\}_h$ on $\mathcal{X}$ Mosco-converges\footnote{We slightly generalise the usual concept of Mosco convergence, where $\mathcal{X}$ should be a Banach space where the weak topology is defined by duality with $\mathcal{X}^*$.} to $\F:\mathcal{X}\to\R\cup\{\infty\}$ as $h\to0$, written as $\F^h\xrightarrow{M}\F$, whenever
	\begin{enumerate}
		\item (\emph{Lower bound}) For any sequence $\rho^h \xrightarrow[h\to0]{\tau_w} \rho$ in $\mathcal{X}$ there holds 
			\begin{equation*}
				\liminf_{h\to0} \F^h(\rho^h) \geq \F(\rho);
			\end{equation*}
		\item (\emph{Recovery sequence}) For all $\rho \in \mathcal{X}$ there is a sequence $\rho^h \xrightarrow[h\to0]{\tau_s} \rho$ in $\mathcal{X}$ such that
			\begin{equation*}
				\limsup_{h\to 0} \F^h(\rho^h) \leq \F(\rho).
			\end{equation*}
	\end{enumerate}
\end{defn}
In this paper we take $\mathcal{X}=\P^\S_2(\R^d)$  (defined in Subsection~\ref{ssec:ldtowasser}), and for $\tau_w$ we take the narrow topology, characterised by narrow convergence:
\begin{align*}
  \rho^h\weakto\rho && \text{if and only if }&& \int\!\phi(x)\,\rho^h(dx)\to\int\!\phi(x)\,\rho(dx) \text{ for all } \phi\in C_b(\R^d).
\end{align*}
For the strong topology $\tau_s$, we take the weakest topology such that all functionals $\rho\mapsto\int\!x^2\,\rho(dx)$, $\rho\mapsto\S(\rho)$ and $\rho\mapsto\int\!\phi(x)\,\rho(dx)$ for all $\phi\in C_b(\R^d)$ are continuous. 
Since this topology is first-countable, convergence in $(\P_2^\S,\tau_s)$ is characterised by convergence in the Wasserstein topology plus convergence of the entropy functional $\S$. In fact, we prove below that convergence in this topology implies strong $L^1$-convergence of the sequence and its entropies. These important facts will be used to prove the Mosco-convergence Theorems~\ref{theo: Fokker-Planck Gamma convergence} and \ref{theo: diff dec gamma convergence}. 
Let $\mathcal{L}^d$ be the $d$-dimensional Lebesgue measure.
\begin{lem}
\label{theo: strong entropy convergence}
   Let $\rho^h\to\rho$ in $\P_2^\S(\R^d)$ in the strong topology, i.e.:
  \begin{align}
    &d(\rho^h,\rho)\to 0,\ \text{ in the Wasserstein metric,} \label{eq:hyp1} \\
    &\S(\rho^h)\to \S(\rho).\label{eq:conv}
  \end{align}
    Then $\rho^h$ and $\rho$ are $\mathcal{L}^d$-absolutely continuous and can be identified with their densities, i.e.	$	\rho^h, \rho \in L^1(\R^d),$ and there is a subsequence such that 
	\begin{align}
		\rho^h&\to\rho, \label{stelling:a}\\
		\rho^h\log\rho^h &\to\rho\log\rho\label{stelling:b}
	\end{align}	
	strongly in $L^1(\R^d)$.
\end{lem}

\begin{proof}[Proof of Lemma \ref{theo: strong entropy convergence}.]
%

\textbf{Step I - Decomposition of the entropy.} To deal with the fact that $\S$ is not bounded from below, we rewrite $\S$ in the following way. Define, for any $\alpha\in \R$ with $\alpha >d$
\begin{align*}
  c^{-1}:=\int_{\R^d}\frac{1}{(1+|x|)^\alpha}dx,  && \nu(dx)=\nu(x)\,dx=\frac{c}{(1+|x|)^\alpha}\,dx,
\end{align*}
and let $\mathcal{H}$ be the relative entropy on two probability measures $\gamma,\nu\in \P(\R^d)$:
\begin{equation}
\label{eq:relative entropy}
	\mathcal{H}(\gamma|\nu) :=	\begin{cases}
  \ds\int \frac{d\gamma}{d\nu}(x)\log\frac{d\gamma}{d\nu}(x)\; \nu(dx),	&\text{if } \gamma \ll \nu, \\
																				\;+\infty,																												&\text{otherwise}.
													\end{cases}
\end{equation}
(Note that $\S(\rho)=\mathcal{H}(\rho|\mathcal{L}^d)$.) Then for any $\rho \in \mathcal{P}_2^\S$, we can write
\begin{align}
	\S(\rho)=\int_{\R^d} \rho \log \rho\, dx &=	\int_{\R^d} \frac{\rho}{\nu} \log\left(\frac{\rho}{\nu}\right) \nu\, dx + \int_{\R^d} \rho \log(\nu)\, dx \nonumber\\
			&=	\mathcal H(\rho|\nu) + \log c -\alpha \int_{\R^d} \rho \log(1+|x|)\, dx. \label{eq:firstlim}
\end{align}	
By \eqref{eq:hyp1} and \cite[Lem.~5.1.7]{Ambrosio2008} 
\begin{equation}
\label{eq:wwwconv}
	\int_{\R^d}\rho^h(x)\phi(x)\, dx \to \int_{\R^d}\rho(x)\phi(x)\, dx 
\end{equation}
for all continuous functions $\phi:\R^d \to\R$ such that $|\phi(x)|\leq A +B|x|^2$ for all $x\in \R^d$, for some $A,B\geq 0$. This implies that the last term on the right of \eqref{eq:firstlim} converges:
\begin{equation}
\label{eq:firstconv}
	\alpha \int_{\R^d} \rho^h(x) \log(1+|x|)\, dx \to \alpha \int_{\R^d} \rho(x) \log(1+|x|)\, dx,
\end{equation}
so that the study of $\S(\rho^h)$ can be reduced to the study of $\mathcal{H}(\rho^h|\nu)$. 

\textbf{Step II - convergence of the plans.} Define the measures $\gamma^h\in \P(\R^d \times \R)$ by
\begin{equation*}
	\int_{\R^d \times \R}\psi(x,y)\, \gamma^h(dx\,dy) = \int_{\R^d}  \psi\left(x,\frac{\rho^h(x)}{\nu(x)}\right) \nu(x)\,dx\qquad \text{for all }\psi \in C_b(\R^d \times \R).
\end{equation*}
The marginals $\pi^1\gamma^h$ and $\pi^2\gamma^h$ then satisfy
\begin{align}
	\int_{\R^d }\phi(x)\, \pi^1\gamma^h(dx) &= \int_{\R^d}\phi(x)\, \nu(x)\,dx , \label{eq:first marginal}\\
		\int_{\R }\varphi(y)\, \pi^2\gamma^h(dy) &= \int_{\R^d}\varphi\left(\frac{\rho^h(x)}{\nu(x)}\right) \nu(x)\,dx\notag,
\end{align}
for all $\phi \in C_b(\R^d)$, for all $\varphi \in C_b(\R)$. 
We claim that
\begin{itemize}
		\item there exists $\gamma\in \P(\R^d \times \R)$ such that, up to subsequences, $\gamma^h\weakto \gamma$ (narrowly);
		\item the barycentric projection \eqref{eq:bary} of the limit $\gamma$, with respect to $\nu$, is $\rho/\nu$.
\end{itemize}
In order to prove the first part of the claim, we note that by \cite[Lem.~5.2.2]{Ambrosio2008}, if the marginals of $\gamma^h$ are tight, then  $\gamma^h$ is also tight, and thus (by \cite[Th.~5.1.3]{Ambrosio2008}) relatively compact, with respect to the narrow topology of $\P(\R^d \times \R)$. 
By \eqref{eq:first marginal} the first marginal does not depend on $h$. For the second marginal we use the following integral condition for tightness (\cite[Rem.~5.1.5]{Ambrosio2008}): \itshape ``if there exists a function $G:\R\to[0,+\infty]$, whose sublevels are compact in $\R$, such that
$$ \sup_{h\in \mathbb N}\int_\R G(y)\, \pi^2 \gamma^h(dy)<+\infty,$$ 
then $\{\pi^2 \gamma^h\}$ is tight." \upshape We can choose, as in \cite[Eq.~(9.4.2)]{Ambrosio2008}, the nonnegative, lower semicontinuous, strictly convex function
\begin{equation*}
	G(s):=\left\{
		\begin{array}{ll}
			s(\log s -1)+1 &\text{if }s>0,\\
			1		&\text{if }s=0,\\
			+\infty &\text{if }s<0,
		\end{array}
		\right.
\end{equation*}
defined on $\R$, and observe that 
$$ \int_\R\! G(y)\, \pi^2 \gamma^h(dy)	=\int_{\R^d}\!G\left( \frac{\rho^h(x)}{\nu(x)}\right)\, \nu(x)\, dx=\mathcal{H}(\rho^h|\nu). $$
The last term is bounded, owing to \eqref{eq:wwwconv}, \eqref{eq:firstlim}, and \eqref{eq:firstconv}. We conclude that $\gamma^h$ is relatively compact and therefore, up to subsequences, $\gamma^h$ converges to a measure $\gamma\in \P(\R^d\times \R)$. 

In order to prove the second part of the claim, note that by disintegration of measures \cite[Th.~5.3.1]{Ambrosio2008}, there exists a family $\{\mu_x\}_{x\in \R^d}\subset \P(\R)$ such that
\begin{equation}
\label{eq:disintegrate}
	\int_{\R^d\times \R}\psi(x,y)\,\gamma(dx\,dy)=\int_{\R^d} \left( \int_\R \psi(x,y)\mu_x(dy)\right)\, \nu(dx)
\end{equation}
for every Borel map $\psi:\R^d\times \R\to [0,+\infty]$. We want to identify the barycentric projection of $\gamma$ with respect to $\nu$, that is, the function
\begin{equation}
\label{eq:bary}
	x\mapsto \int_\R y\, \mu_x(dy),
\end{equation}
with $\rho/\nu$. This can be done if we can choose as test function $\psi$ a function of the form $(x,y)\to\phi(x)y$, with $\phi\in C_b(\R^d)$. Since such a function is not bounded, we first need to check that it is uniformly integrable. Since $\mathcal{H}(\rho^h|\nu)$ is bounded, there is a constant $C_1>0$ such that, for all $R>1$,
\begin{align*}	
	C_1&>\sup_h \int_{\R^d}\!G\left( \frac{\rho^h(x)}{\nu(x)}\right)\, \nu(x)\, dx\\&\geq \sup_h \int_{\{\rho^h > R\}}\!G\left( \frac{\rho^h(x)}{\nu(x)}\right)\, \nu(x)dx\\
	& = \sup_h \int_{\{\rho^h > R\}} \rho^h(x) \log\left(\rho^h(x)\frac{(1+|x|)^\alpha}{c}\right) dx\\	
	&\geq \sup_h \int_{\{\rho^h > R\}} \rho^h(x)\log R-\rho^h\log c +\alpha\rho^h(x)\log(1+|x|) dx\\
	&\geq \log (R)\, \sup_h \int_{\{\rho^h > R\}}\rho^h(x)\, dx -\log c -\sup_h \alpha \int_{\R^d}\rho^h(x)\log(1+|x|) dx\\ 
	&\stackrel{\eqref{eq:firstconv}}{\geq} \log(R) \sup_h \int_{\{\rho^h > R\}}\rho^h(x)\, dx -C_2.	
\end{align*}
Therefore, 
\begin{equation}
\label{eq:rhohui}
	\lim_{R\to \infty} \sup_h \int_{\{\rho^h > R\}} \rho^h dx \leq \lim_{R\to \infty} \frac{C_1+C_2}{\log(R)}=0,
\end{equation}
i.e., $\rho^h$ is uniformly integrable. Since for every $\phi\in C_b(\R^d)$
\begin{align*} 
	\lim_{R\to \infty} \sup_h \int_{\{\phi(x)y\geq R\}} \phi(x)y\, \gamma^h(dx\,dy) &\leq \lim_{R\to \infty} \sup_h \|\phi\|_\infty\int_{\{|y|\geq R/\|\phi\|_\infty\}} y\, \gamma^h(dx\,dy)\\
	&= \lim_{R\to \infty} \sup_h \|\phi\|_\infty\int_{\{\rho^h \geq R/\|\phi\|_\infty\}} \rho^h(x)\, dx\stackrel{\eqref{eq:rhohui}}{=}0,
\end{align*} 
we conclude that the function $\R^d\times \R \ni (x,y) \mapsto \phi(x)y \in \R$ is uniformly integrable with respect to the measures $\{\gamma^h\}$. Uniform integrability, owing to \cite[Lem.~5.1.7]{Ambrosio2008}, yields 
\begin{equation*}
	\lim_{h\to \infty} \int_{\R^d\times \R}\phi(x)y\,\gamma^h(dx\,dy)= \int_{\R^d\times \R}\phi(x)y\,\gamma(dx\,dy)\stackrel{\eqref{eq:disintegrate}}{=}\int_{\R^d} \phi(x) \left(\int_\R y\, \mu_x(dy)\right) \nu(dx).
\end{equation*}
On the other hand, by \eqref{eq:conv} we know that 
$$	\lim_{h\to \infty} \int_{\R^d\times \R}\phi(x)y\,\gamma^h(dx\,dy)= \lim_{h\to \infty} \int_{\R^d}\phi(x)\frac{\rho^h(x)}{\nu(x)}\, \nu(dx)=\int_{\R^d}\phi(x)\frac{\rho(x)}{\nu(x)}\, \nu(dx).
$$
We conclude that the weak limit of the densities is equal to the barycentric projection of the limit plans:
\begin{equation}
\label{eq:weakeq}
	\frac{\rho(x)}{\nu(x)}=\int_\R y\, \mu_x(dy)\qquad \text{for a.e. }x\in\R^d.
\end{equation}
\textbf{Step III - pointwise convergence.} We compute 
\begin{align}
	\liminf_{h\to \infty} \mathcal{H}(\rho^h|\nu) &=\liminf_{h\to \infty} \int_{\R^d}\!G\left( \frac{\rho^h(x)}{\nu(x)}\right)\nu(dx) \nonumber\\
	&=\liminf_{h\to \infty} \int_{\R^d \times \R}\!G(y)\,  \gamma^h(dx\,dy) \nonumber\\
		&\geq \int_{\R^d \times \R}\!G(y)\,  \gamma (dx\,dy) \nonumber\\
	&= \int_{\R^d}\left( \int_\R\! G(y)\mu_x(dy)\right)\,\nu(dx)\nonumber\\
	&  \geq \int_{\R^d}\!G\left( \int_\R y\, \mu_x(dy)\right)\,\nu(dx)\label{eq:jensen}\\
	&= \int_{\R^d}\!G\left( \frac{\rho(x)}{\nu(x)}\right) \nu(dx)=\mathcal{H}(\rho|\nu),\label{eq:final}
\end{align}
where, in the last three steps, we used \eqref{eq:disintegrate}, Jensen's inequality, and \eqref{eq:weakeq}. Collecting all the computations we have
\begin{align*}
		\mathcal{H}(\rho|\nu)&\stackrel{\eqref{eq:firstlim}}{=}\S(\rho)-\log c +\alpha \int_{\R^d} \rho(x) \log(1+|x|)\, dx \\
									& \stackrel{\eqref{eq:conv},\eqref{eq:firstconv}}{=}\lim_{h\to \infty} \left\{\S(\rho^h)-\log c +\alpha \int_{\R^d} \rho^h(x) \log(1+|x|)\, dx\right\}\\
									&\stackrel{\eqref{eq:firstlim}}{=} \liminf_{h\to \infty} \mathcal{H}(\rho^h|\nu)\\
									&\stackrel{\eqref{eq:final}}{\geq} \mathcal{H}(\rho|\nu).
\end{align*}		
Therefore, the inequality in \eqref{eq:jensen} must be an equality, which, by strict convexity of $G$, implies that $\mu_x$ is a Dirac delta concentrated in $\frac{\rho(x)}{\nu(x)}$, for a.e. $x\in \R^d$. As a consequence
\begin{equation*}
		\frac{\rho^h(x)}{\nu(x)} \to \frac{\rho(x)}{\nu(x)}\qquad \text{for a.e. }x\in \R^d,
\end{equation*}
and therefore
\begin{equation}
\label{eq:conclusion}
		\rho^h(x) \to \rho(x) \qquad \text{for a.e. }x\in \R^d.
\end{equation}
\textbf{Step IV - strong convergence.} To prove the strong convergence results~\eqref{stelling:a} and \eqref{stelling:b}, recall the following theorem from \cite[Th.~1]{Brezis1983} for any measure $\kappa$ on $\R^d$ and non-negative $\rho^h,\rho\in L^1(\kappa)$:
\begin{equation}
\label{th:brezis}
\text{If } \int\!\rho^h\,d\kappa \to \int\!\rho\,d\kappa \text{ and } \rho^h(x)\to \rho(x)\quad\kappa\text{-a.e.}, \text{ then } \rho^h\to \rho \text{ strongly in } L^1(\kappa).
\end{equation}
Clearly, \eqref{stelling:a} follows from \eqref{eq:conclusion} and \eqref{th:brezis} by taking $\kappa=\mathcal{L}^d$.


In order to prove \eqref{stelling:b},  let $G^h\!:=\!G\left( \rho^h/\nu\right)$, $G^0\!:=\!G\left( \rho/\nu\right)$. Since $G$ is continuous and $\rho^h\to\rho$ almost everywhere, 
\begin{equation}
\label{eq:hpw}
		G^h(x)\to G^0(x)\qquad \text{for a.e. }x\in \R^d.
\end{equation}
Moreover, from the proof of \eqref{stelling:a}, we know that
\begin{equation}
\label{eq:lastpage}
	\int_{\R^d}\!G^h(x)\, \nu(dx)=\mathcal{H}(\rho^h|\nu) \to \mathcal{H}(\rho|\nu)=\int_{\R^d}\!G^0(x)\, \nu(dx). 
\end{equation}
Again by \eqref{th:brezis}, now with $\kappa=\nu$, it follows from \eqref{eq:hpw} and \eqref{eq:lastpage} that $G^h\to G^0$ strongly in $L^1(\nu)$. Therefore, because the density of $\nu$ is uniformly bounded
\begin{equation}
\label{eq:strongh}
	G^h \to G^0\qquad \text{strongly in }L^1(\R^d).
\end{equation}
It now follows from \eqref{stelling:a} and \eqref{eq:strongh} together with
\begin{align*} 
	\rho^h \log \rho^h &= G^h \nu +\rho^h\log(\nu)+\rho^h -\nu\\
			&=G^h f +\rho^h(\log(c)+1) -\alpha\rho^h\log(1+|\cdot|) -\nu
\end{align*}			
that, in order to prove \eqref{stelling:b} we only need to check that
$$ \rho^h\log(1+|\cdot|) \to \rho\log(1+|\cdot|)\qquad \text{strongly in }L^1(\R^d).$$
This  follows from the uniform integrability of the first moments of $\rho^h$ and from the strong $L^1$-convergence of $\rho^h$. Precisely, since $d(\rho^h,\rho)\to 0$, then $\rho^h$ has uniformly integrable $p$-moments for all $p\in(0,2)$. In particular, for every $\varepsilon>0$ there exists $R_\varepsilon>0$ such that
$$ \sup_h \int_{|x|\geq R_\varepsilon} |x|\rho^h(x)\, dx \leq \varepsilon.$$
For all $\varepsilon>0$ we estimate
\begin{align*}
		\int_{\R^d} \Big|\rho^h(x) \log(1+|x|) - \rho(x) \log(1+|x|)\Big| dx &\leq \int_{|x|<R_\varepsilon} \big|\rho^h(x) -\rho(x)\big| \log(1+|x|)\,dx \\
		&\hspace{-2cm} + \int_{|x|\geq R_\varepsilon} |x|\rho^h(x)\,dx +\int_{|x|\geq R_\varepsilon} |x|\rho(x)\,dx\\
		&\hspace{-2cm} \leq \|\rho^h-\rho\|_{L^1}\log(1+R_\varepsilon) + 2\varepsilon 
\end{align*}
and therefore, for all $\varepsilon>0$ 
$$\lim_{h\to \infty}		\int_{\R^d} \Big|\rho^h \log(1+|x|)dx - \rho \log(1+|x|)\Big| dx\leq 2\varepsilon.$$
By the arbitrariness of $\varepsilon$, we conclude strong $L^1$-convergence.
\end{proof}

\section{Diffusion with drift}
\label{sec: FP}
In this section we discuss the case of diffusion with drift but without decay ($\Psi\not\equiv0$, $\lambda=0$), i.e.\ equation~\eqref{eq: Fokker-Planck}. First we describe the particle system that we use as a microscopic model for this equation, and derive the corresponding large-deviation principle. Next, we show that the large-deviation rate functional relates to the energy-dissipation functional~\eqref{def: KFP} in a Mosco-convergence sense.

\subsection{Microscopic model}
Consider  a system of $n$ independent (i.e. non-interacting) point particles in $\R^d$. We wish $\init\rho \in \P(\R^d)$ to represent the distribution of initial positions, and implement this as in~\cite{Leonard2007}.  For each $n$ choose $x_i\in \R^d, 1\leq i \leq n$ such that 
\begin{equation*}
	\frac{1}{n} \sum_{i=1}^n \delta_{x_i} \longweakto \init\rho \quad \text{ as } n\to \infty.
\end{equation*}
We then set the (deterministic) initial position\footnote{This way of enforcing the initial distribution~$\init\rho$ is different from the approach  of~\cite{Adams2011}. It provides a more direct result, and is easier to interpret; see Remark~\ref{rem: quenched LDP} for a discussion.} of particle $i\in\{1,\dots,n\}$ to be $x_i$.

The dynamics of the system is determined by the probability for particle $i$ to move from $x_i$ to a (random) position $Y^h_i$ in some fixed time $h>0$. We take this transition probability to be the fundamental solution $\eta^t(y;x)$ of the drift-diffusion equation~\eqref{eq: Fokker-Planck}, in the following sense:
\begin{defn}\label{def: fundamental solution} We say that a mapping $\eta: \R^d \times [0,\infty) \to \P(\R^d)$ is a \emph{fundamental solution} of the Fokker-Planck equation \eqref{eq: Fokker-Planck} whenever
	\begin{enumerate}
 		\item $\eta^{x,t}(B)$ is measurable in $x \in \R^d$ and $t \in [0,\infty)$ for all fixed Borel sets $B\subset\R^d$,
 		\item for all $\phi \in C_b^{2,1} (\R^d\times[0,\infty))$ and $(x,T) \in \R^d\times[0,\infty)$ there holds:
			\begin{equation*}
				\int_0^T \! \int \! \left( \partial_t \phi + \lapl \phi - \grad\Psi\cdot\grad \phi \right) \, \eta^{x,t}(dy)\, dt = \int\! \phi(y,T)\, \eta^{x,T}(dy) - \phi(x,0).
			\end{equation*}
	\end{enumerate}
\end{defn}
\noindent If we assume that $\Psi \in C_b^2(\R^d)$, that is $\Psi\in C^2(\R^d)$ and $|\Psi|, |\grad\Psi|$, and $|\lapl\Psi|$ are all bounded, then there exists an absolutely continuous fundamental solution with a density in $C^{2,1}(\R^d\times(0,\infty))$ \cite[Th.~1.10]{Friedman1964}. We can thus identify this fundamental solution $\eta^{x,t}$ with its density $\eta^t(\,\cdot\,;x)$. 

Using this fundamental solution as the transition probability, the empirical measure $L_n^h=n^{-1}\sum_{i=1}^n \delta_{Y^h_i}$ will converge almost surely to $\init\rho\ast\eta^h$, which is the solution to \eqref{eq: Fokker-Planck} at time $h$ with initial condition $\init\rho$ \cite[Th. 11.4.1]{Dudley1989}. In this sense the proposed system is indeed a  microscopic precursor of this equation.

\subsection{From large deviations to Wasserstein gradient flow}
The sequence $L_n^h$ satisfies a large-deviation principle with rate $n$ and rate functional (see Corollary~\ref{coro: contracted quenched LDP} in the Appendix):
\begin{equation}
\label{def: Fokker-Planck rate functional}
	\J^h\FP(\rho|\init\rho):= \inf_{q \in \Gamma(\init\rho,\rho)} \mathcal{H}\!\left( q | \init\rho \,\eta^h \right)\!,
\end{equation}
where $\mathcal{H}$ is the relative entropy~\eqref{eq:relative entropy} on $\P(\R^d\times \R^d)$, and, by abuse of notation we write $(\init\rho\, \eta^h)(dx\, dy) = \init\rho(x) \eta^h(y;x) \, dx \, dy$.

We now prove the following relationship between this rate functional $\J^h\FP$ and the gradient-flow functional $\K^h\FP$ (given by \eqref{def: KFP}):
\begin{theo}
\label{theo: Fokker-Planck Gamma convergence}	Assume that Conjecture~\ref{conj} holds, and that $\Psi \in C_b^2(\R^d)$. Then for any $\init\rho\in \P_2^\S(\R^d)$
	\begin{equation}
		\begin{split}
			\J^h\FP(\,\cdot\,|\init\rho) - \frac{1}{4h}d^2(\init\rho, \,\cdot\,)  \xrightarrow[h\to 0]{M}\; & \tfrac12 \S(\cdot) - \tfrac12 \S(\init\rho) + \tfrac12 \E(\cdot) - \tfrac12 \E(\init\rho),\\
			& = \K^h\FP(\,\cdot\,|\init\rho) - \frac{1}{4h}d^2(\init\rho, \,\cdot\,).
		\end{split}
	\end{equation}
\end{theo}
\noindent The proof relies heavily on an estimate of the fundamental solution $\eta^h$. To explain this estimate morally, observe that if $\Psi$ is affine, i.e.\ $\Psi(x) = c\cdot x$, then the force field $\grad \Psi$ is homogeneous, leading to constant drift $c$. In this simple case, the fundamental solution can be written explicitly:
\begin{equation}
\label{eq: easy fundamental solution}
	\eta^t(y;x) = \frac1{(4\pi t)^{d/2}}\; e^{-|y-(x-ct)|^2/4t} 
	= \theta^t(y-x) e^{-\tfrac12 c\cdot y + \tfrac12 c\cdot x - \tfrac14 |c|^2 t},
\end{equation}
where $\theta^t$ is again the diffusion kernel~\eqref{def: diffusion kernel}. Although for an arbitrary $\Psi$ an analytic expression for the fundamental solution is generally difficult to find, the expression \eqref{eq: easy fundamental solution} above suggests that it can be estimated by something similar for small times. Below we see that this is indeed the case. We expect that this estimate is not a new result, but since we haven't been able to find it in the literature we include the proof here for completeness\footnote{See for example \cite{Aronson1967} for a similar, but not strong enough result.}
\begin{lem}
	Assume $\Psi \in C_b^2(\R^d)$, and let $\eta$ be the fundamental solution from Definition~\ref{def: fundamental solution}. Then there are $\beta_0, \beta_1 \in \R$ such that for every $t>0$:
  \begin{equation}
  \label{eq: diffusion drift fundamental estimate}
		\theta^t(y-x) e^{-\tfrac12 \Psi(y) + \tfrac12 \Psi(x)+ \beta_0 t} \leq \eta^t(y;x) \leq \theta^t(y-x) e^{-\tfrac12 \Psi(y) + \tfrac12 \Psi(x)+ \beta_1 t}
  \end{equation}
   for almost every $x,y \in \R^d$.
\end{lem}

\begin{proof} For brevity we assume that $x=0$ and $\Psi(0) \equiv 0$, and we omit the dependence on~$x$. For $\beta \in \R$ define:
	\begin{equation*}
		\zeta_\beta (y,t) := \eta^t(y) - \theta^t(y) e^{-\tfrac12\Psi(y) + \beta t}.
	\end{equation*}

	By partial integration we obtain for all $0<\epsilon<T$ and $\phi \in C_b^{2,1}(\R^d \times [\epsilon,T])$:
	\begin{multline}
	\label{eq: weak PDE qbeta}
		\int_\epsilon^T \! \int \! \left (\partial_t \phi(y,t) + \lapl \phi(y,t) - \grad \Psi(y)\cdot \grad \phi(y,t) \right) \zeta_\beta(y,t) \, dy \, dt  \\
			= \int_\epsilon^T\!\int\! \phi(y,t) f_\beta(y,t) \, dy \, dt + \int \! \phi(y,T) \zeta_\beta(y,T) \, dy - \int \! \phi(y,\epsilon) \zeta_\beta(y,\epsilon)\, dy 
  \end{multline}
  with:
	\begin{equation*}
		f_\beta(y,t) := \left(- \tfrac12 \lapl \Psi(y) + \frac{1}{4} |\grad \Psi(y)|^2 + \beta \right) \theta^t(y) e^{-\tfrac12 \Psi(y) + \beta t}.
	\end{equation*}
	Because $\grad \Psi$ and $\lapl \Psi$ are bounded, there are $\beta_0, \beta_1 \in \R$ such that:
	\begin{equation}
	\label{eq: double inequality}
		f_{\beta_0}(y,t) \leq 0 \leq f_{\beta_1}(y,t).
	\end{equation}

	First we exploit this inequality for $\beta_1$. Let $\phi$ be the solution of the adjoint problem:
	\begin{equation}
	\label{eq: diffusion drift adjoint problem}
		-\partial_t \phi = \lapl \phi - \grad \Psi\cdot \grad \phi 
	\end{equation}
	with end condition:
	\begin{equation*}
		\phi^T(y):= H(\zeta_{\beta_1}\!(y,T)),
	\end{equation*}	
	where $H$ is the Heaviside function. Again by \cite[Th.~1.10]{Friedman1964} there exists a positive fundamental solution $\eta^\ast$ and hence a positive bounded solution $\phi \in C^{2,1}(\R^d\times[0,T))$ to \eqref{eq: diffusion drift adjoint problem}. However, \eqref{eq: weak PDE qbeta} requires the test functions to be in $C_b^{2,1}(\R^d\times(0,T])$. To this aim we approximate $\phi$ in the following way. First, let $\phi^T_n$ be a sequence in $C_0^\infty(\R^d)$ such that
	\begin{equation*}
		\phi^T_n \to \phi^T \text{ weakly-$\ast$ in } L^\infty(\R^d).
	\end{equation*}
	Next, let $\phi_n \in C_b^{2,1}(\R^d\times [0,T])$ be the solution of \eqref{eq: diffusion drift adjoint problem} with approximated end condition $\phi^T_n$. For this sequence \eqref{eq: weak PDE qbeta} becomes:
	\begin{eqnarray}
		0 &=& 	\int_\epsilon^T\!\int \phi_n(y,t) f_{\beta_1}\!(y,t) \, dy \, dt + \int\! \phi^T_n(y)\zeta_{\beta_1}\!(y,T) \, dy - \int\! \phi_n(y,\epsilon) \zeta_{\beta_1}\!(y,\epsilon) \, dy \notag\\
		&\xrightarrow[\epsilon \to 0]{(i.)}&	\int_0^T\!\int \phi_n(y,t) f_{\beta_1}\!(y,t) \, dy \, dt + \int\! \phi^T_n(y)\zeta_{\beta_1}\!(y,T) \, dy \notag\\
		&\xrightarrow[n \to \infty]{(ii.)}& \int_0^T\!\int \phi(y,t) f_{\beta_1}\!(y,t) \, dy \, dt + \int\! H\!\left(\zeta_{\beta_1}\!\left(y,T\right)\right) \zeta_{\beta_1}\!(y,T) \, dy,
			\label{eq: double approximation result} 
	\end{eqnarray}
	using properties $(i.)$ and $(ii.)$ that we will prove below. From this we infer for the positive part of $\zeta_{\beta_1}$: 
	\begin{equation*}
		0 \leq \int\! \zeta^+_{\beta_1}\!(y,T)\, dy \stackrel{\eqref{eq: double approximation result}}{=} - \int_0^T \! \int \! \underbrace{\phi(y,t)}_{\geq 0} \underbrace{ f_{\beta_1}\!(y,t)}_{\geq 0} \, dy\, dt \leq 0.
	\end{equation*}

	Analogously we use the other inequality from \eqref{eq: double inequality} and conclude that for all $T>0$:
  \begin{equation*}
  	\zeta_{\beta_1}\!(y,T) \leq 0 \leq \zeta_{\beta_0}(y,T) \,\,\,  \text{ for almost every } y \in \R^d,
  \end{equation*}
  which proves the statement. 
  
  We still owe the reader the proof of the two approximations in \eqref{eq: double approximation result}. 
  \begin{enumerate}
  	\item[(i.)]	The argument follows from $\zeta_{\beta_1}\!(x,\epsilon) \to 0$ weakly in $L^1(\R^d)$ as $\epsilon \to 0$. Then for any fixed $n$:
				 	\begin{multline*}
						\left| \int\!\left(\phi_n(y,\epsilon)-\phi_n(y,0)\right) \zeta_{\beta_1}\!(y,\epsilon)\, dy \right| = \left| \int\! \int\limits_0^\epsilon\! \partial_t \phi_n(y,t) \, dt \, \zeta_{\beta_1}\!(y,\epsilon) \, dy \right| \\
						\leq \underbrace{\epsilon}_{\to 0} \underbrace{\left\| \partial_t \phi_n \right\|_{L^\infty(\R^d \times [0,T])}}_{\text{bounded}} \underbrace{ \left| \int\! \zeta_{\beta_1}\!(y,\epsilon)\, dy \right|}_{\to 0}\xrightarrow[\epsilon\to 0]{} 0.
					\end{multline*}
					Hence:
					\begin{multline*}
						\int\! \phi_n(y,\epsilon) \zeta_{\beta_1}\!(y,\epsilon)\,dy	\\
							= \int\! \left(\phi_n(y,\epsilon)-\phi_n(y,0)\right) \zeta_{\beta_1}\!(y,\epsilon)\, dy + \int\! \phi_n(y,0) \zeta_{\beta_1}\!(y,\epsilon)\, \xrightarrow[\epsilon\to 0]{} 0.
					\end{multline*}
  	\item[(ii.)] For the second convergence in \eqref{eq: double approximation result}, we can assume that the approximation of the end condition satisfies:
  				\begin{equation*}
  					0 \leq \phi_n^T(y) \leq \phi^T(y) \,\, \text{ for all } y \in \R^d.
  				\end{equation*}
  				Therefore:
  				\begin{equation*}
  					\big| \phi_n(y,t) f_{\beta_1}\!(y,t) \big| \leq \big| \phi(y,t) f_{\beta_1}\!(y,t) \big| \leq \underbrace{\| \phi^T \|_{L^\infty(\R^d)} \big| f_{\beta_1}\!(y,t) \big|}_{\in L^1(\R^d\times(0,T))}.
  				\end{equation*}
					Since for the fundamental solution $\eta^\ast$ of the adjoint problem \eqref{eq: diffusion drift adjoint problem} there holds $z \mapsto {\eta^\ast}^t (y,z) \in L^1(\R^d)$, we have:
  				\begin{equation*}
  					\phi_n(y,t) = \int \! {\eta^\ast}^t (y,z) \phi_n^T(z) \, dz \xrightarrow[n\to\infty]{} \int\! {\eta^\ast}^t(y,z)\phi^T(z)\, dz = \phi(y,t)
  				\end{equation*}
  				pointwise. The Dominated Convergence Theorem then gives
  				\begin{equation*}
  					\phi_n f_{\beta_1} \xrightarrow[n\to\infty]{L^1} \phi f_{\beta_1}.
  				\end{equation*}
	\end{enumerate}  
\end{proof}

Observe that the factors $1/2$ in the exponent of \eqref{eq: diffusion drift fundamental estimate} correspond to the factors $1/2$ of the energy in expression \eqref{def: KFP}.
We are now ready to prove the Mosco-convergence result.
\begin{proof}[Proof of Theorem~\ref{theo: Fokker-Planck Gamma convergence}] To prove the lower bound, take any sequence $\rho^h \weakto \rho$ in $\P_2^\S(\R^d)$ and calculate
\begin{align*}
	&	\hskip-1cm\lefteqn{\liminf_{h\to 0} \,\J^h\FP(\rho^h|\init\rho) - \tfrac{1}{4h}d^2(\init\rho, \rho^h)} \\
		&\stackrel{\eqref{def: Fokker-Planck rate functional}}{\hq =\hq} \liminf_{h\to 0} \inf_{q\in\Gamma(\init\rho,\rho^h)} \mathcal{H}(q|\init\rho \eta^h) - \tfrac{1}{4h}d^2(\init\rho, \rho^h) \\
		 				&\stackrel{\eqref{eq: diffusion drift fundamental estimate}}{\hq\geq\hq} \liminf_{h\to 0} \inf_{q\in\Gamma(\init\rho,\rho^h)} \! \mathcal{H}(q|\init\rho \theta^h) \\
						& \qquad\qquad - \iint \!\left(-\tfrac12 \Psi(y) + \tfrac12 \Psi(x)+ \beta_1 h\right)q(dx\,dy)- \tfrac{1}{4h}d^2(\init\rho, \rho^h) \\
		 				&\hq=\hq \liminf_{h\to 0} \inf_{q\in\Gamma(\init\rho,\rho^h)} \! \mathcal{H}(q|\init\rho \theta^h) - \tfrac{1}{4h}d^2(\init\rho,\rho^h) + \tfrac12 \E(\rho^h) - \tfrac12 \E(\init\rho) - \beta_1 h \\
		 				&\hq\geq\hq \tfrac12 \S(\rho) -\tfrac12 \S(\init\rho) + \tfrac12 \E(\rho) - \tfrac12 \E(\init\rho),
	\end{align*}
where the last inequality follows from Conjecture~\ref{conj} and the (narrow) continuity of $\rho\mapsto \E(\rho)$.
	
\smallskip
To construct a recovery sequence, fix a $\rho \in \P_2^\S(\R^d)$ and take a recovery sequence $\rho^h\to\rho$ from Conjecture~\ref{conj}, in the strong topology of $\P_2^\S(\R^d)$. Then similarly:
\begin{equation*}
\begin{split}
\limsup_{h\to 0} {}&\J^h\FP(\rho^h|\init\rho) - \tfrac{1}{4h}d^2(\init\rho, \rho^h) 
\stackrel{\eqref{def: Fokker-Planck rate functional}}{\hq =\hq}
  \limsup_{h\to 0} \inf_{q\in\Gamma(\init\rho,\rho^h)} \mathcal{H}(q|\init\rho \eta^h) -
  \tfrac{1}{4h}d^2(\init\rho, \rho^h) \\
&\stackrel{\eqref{eq: diffusion drift fundamental estimate}}{\hq\leq\hq} \limsup_{h\to 0} \inf_{q\in\Gamma(\init\rho,\rho^h)}	
  \mathcal{H}(q|\init\rho \theta^h) - \tfrac{1}{4h}d^2(\init\rho, \rho^h) 
    + \tfrac12 \E(\rho^h) - \tfrac12 \E(\init\rho) - \beta_0 h \\
&\hq\leq\hq\tfrac12 \S(\rho) -\tfrac12 \S(\init\rho) + \tfrac12 \E(\rho) 
  - \tfrac12 \E(\init\rho).
\end{split}
\end{equation*}
\end{proof}

\section{Diffusion with drift and decay}
\label{sec: diff dec} 

In this section we discuss the case of diffusion with decay. For brevity, we first consider the case without drift ($\Psi\equiv0$, $\lambda>0$). First we describe the particle system that we use as a microscopic model for this equation, and calculate the corresponding large-deviation principle. We proceed with the main results for this equation: Mosco-convergence to an energy-dissipation functional, and convergence of the approximation scheme to the solution of the diffusion-decay equation. Finally, we discuss how the system can be generalised to include drift, and how the decay can be generalised to diffusion-reaction equations.

\subsection{Microscopic model}
\label{sec: diff dec micro model}

In contrast to the case without decay, the diffusion-decay equation \eqref{eq: diffusion decay} is not mass-conserving, implying that the Wasserstein distance between two time instances of a solution is not defined. To overcome this difficulty, we assume that all decayed matter continues to exist after its decay, but in a different form. We thus distinguish between \emph{normal}, non-decayed matter, denoted by $N$, and \emph{decayed} or \emph{dark matter}, denoted by $D$.

The microscopic model now consists of a finite number $n$ of independent non-interacting point particles moving in $\R^d\times\{N,D\}$. Similarly to the non-decaying model, we fix an initial distribution $\init\rho \in \P(\R^d\times\{N,D\})$ and initial positions $x_i\in\R^d$ and states $\mu_i \in \{N,D\}$ such that: 
\begin{align*}
	&\frac{1}{n} \sum_{\substack{i=1\\\mu_i=N}}^n \delta_{x_i}\longweakto \init\rho_N \quad\text{ and }\quad \frac{1}{n} \sum_{\substack{i=1\\\mu_i=D}}^n \delta_{x_i}\longweakto \init\rho_D & \text{ as } n\to \infty.
\end{align*}
For the dynamics of the system we assume that the motion of all particles in $\R^d$ is independent of their motion in $\{N,D\}$ (this construction will yield separate terms in the rate functional for both processes). We take the motion in $\R^d$ during some fixed time step $h>0$ to be Brownian, ie.\ governed by the transition probability $\theta^h$ from \eqref{def: diffusion kernel}. For the motion in $\{N,D\}$,  {we assume that the time after which a particle changes from $N$ to $D$ is exponentially distributed with rate $\lambda$. Since decay is a one-way street, the probability for a particle to change back from $D$ to $N$ is zero. } This results in a probability for a particle to change from state $\mu$ to $\nu$ during the time step $h$ of
\begin{equation*}
	r^h_{\mu\nu} :=	\begin{cases}	e^{-\lambda h}, 		&\mu=N, \nu=N	\\ 
													1 - e^{-\lambda h},	&\mu=N, \nu=D	\\ 
													0,									&\mu=D, \nu=N \\
													1,									&\mu=D,	\nu=D.
						\end{cases}
\end{equation*}

Denote $L_n^h := n^{-1} \sum_{i=1}^n \delta_{(Y^h_i,\nu^h_i)}$, where $Y^h_i \in \R^d$ and $\nu^h_i \in \{N,D\}$ are the random position and state of the $i^\text{th}$ particle at time $h$. Indeed, $L_n^h$ converges almost surely to the solution at time $h$ of the system \cite[Th. 11.4.1]{Dudley1989}
\begin{equation}
\label{def:diff dec system}
	\begin{cases}	\partial_t u_N = \lapl u_N - \lambda u_N,	&\R^d \times (0,\infty),\\
								\partial_t u_D = \lapl u_D + \lambda u_N,	&\R^d \times (0,\infty)
	\end{cases}
\end{equation}
with initial condition $(\init\rho_N,\init\rho_D)$. In this sense, the thus defined particle system is a microscopic interpretation of the diffusion-decay equation \eqref{eq: diffusion decay} (if we ignore the dark matter).

\subsection{Large deviations to gradient flow to PDE}

While the inspiration for this paper was equation~\eqref{eq: diffusion drift decay}, the construction above suggests to consider not only~\eqref{eq: diffusion drift decay} but also the augmented system of equations~\eqref{def:diff dec system} (and its extensions to non-zero $\Psi$). For this reason we derive a large-deviation principle and a corresponding energy-dissipation functional for this system, and afterwards simplify by contraction, leading to results for~\eqref{eq: diffusion drift decay}.

Let $M^h_n := n^{-1}\sum_{i=1}^n \delta_{(x_i,\mu_i, Y^h_i,\nu^h_i)}$ be the empirical measure of the initial and final configurations corresponding to the particle system defined above. Then (see Theorem~\ref{theo: pair measure LDP}) the sequence $M^h_n$ satisfies a large-deviation principle in $\P(\R^d\times\{N,D\}\times\R^d\times\{N,D\})$ with rate $n$ and rate functional
\begin{equation*}
	\begin{cases}
		\displaystyle
		\ \smash{\sum_{\substack{\mu=N,D\\\nu=N,D}}}\mathcal{H}(q_{\mu\nu}|\init\rho_\mu r^h_{\mu\nu}\theta^h),	&	\text{if } q(\,\cdot \times \{N\} \times \R^d \times \{N,D\}) = \init\rho_N(\cdot)\\																			& \quad \text{and } q(\,\cdot \times \{D\} \times \R^d \times \{N,D\}) = \init\rho_D(\cdot),\\[3\jot]
		\ \infty, &	\text{otherwise},	
	\end{cases}
\end{equation*}
writing $q_{\mu\nu}(dx\,dy) = q(dx \times \{\mu\} \times dy \times \{\nu\})$. We note that  definitions \eqref{def: entropy} and \eqref{eq:relative entropy} indeed allow for non-negative Borel measures that are not necessarily probability measures. 

In contrast to the previous case without decay, the special structure of the decay forces us to keep track of more information: not only of the total amount of  dark matter, but of both the pre-existing dark matter and the normal matter that is converted to dark matter in the present time step, separately. We thus obtain a large-deviation principle for the triple empirical measures $\tfrac1n \sum_{i=1}^n \delta_{(\mu_i, Y^h_i,\nu^h_i)}$ with rate $n$ and rate functional (the subscript stands for `Diffusion equation with Decay')
\begin{multline}
\label{def:diff dec rate functional}
	\J^h\DDec(\rho\NN,\rho\ND,\rho\DD|\init\rho_N,\init\rho_D) \; := \inf\Big\{ \displaystyle\sum_{\mu\nu=N\!N,N\!D,D\!D} \inf_{q_{\mu\nu}\in\Gamma(\init\rho_{\mu\nu}\!,\rho_{\mu\nu})}\mathcal{H}\!\left(q_{\mu\nu}\vert \init\rho_\mu r^h_{\mu\nu} \theta^h\right) : \\
 \init\rho\NN,\init\rho\ND \in \mathcal{M}^+(\R^d) \text{ such that } \init\rho\NN+\init\rho\ND=\init\rho_N \Big\}.
\end{multline}
Here $\rho_{\mu\nu}$ is the final-time matter of type $\nu$ that was initially of type $\mu$, and similarly $\init\rho_{\mu\nu}$ is that part of the initial distribution $\init\rho_\mu$ that will become of type $\nu$ at time $h$ (see Figure~\ref{fig: split measures}). Observe that the term $\mathcal{H}(q_{D\!N}|0)$ is zero if and only if $q_{D\!N} \equiv 0 \ae$, and $\infty$ otherwise; indeed no mass is allowed to change from $D$ to $N$. Hence we omit the dependency on $\rho_{D\!N}$.

\begin{figure}[h]
\centering
\begin{tikzpicture}[scale=1, >= latex, decoration={brace, amplitude=4pt}]
  \draw[decorate, thick] (1.2,1.1) -- node[midway, left=3pt]{$\init\rho_N$} (1.2,2.9);
  \draw[decorate, thick] (1.2,0.1) -- node[midway, left=3pt]{$\init\rho_D$} (1.2,0.9);
  \draw[decorate, thick] (7.2,2.9) -- node[midway, right=3pt]{$\rho_N$} (7.2,2.1);
  \draw[decorate, thick] (7.2,1.9) -- node[midway, right=3pt]{$\rho_D$} (7.2,0.1);
	\draw (1.5, 2)[fill=gray!5]	rectangle node[anchor=center] {$\init\rho\NN$}(3, 3);
	\draw (1.5, 1)[fill=gray!5]	rectangle node[anchor=center] {$\init\rho\ND$} (3, 2);
	\draw (1.5, 0)[fill=gray!30]	rectangle node[anchor=center] {$\init\rho\DD$} (3, 1);
	\draw[->, thick] (3,2.5)-- node[anchor=south]{$q\NN$} (5.5,2.5);
	\draw[->, thick] (3,1.5)-- node[anchor=south]{$q\ND$}(5.5,1.5);	
	\draw[->, thick] (3,0.5)-- node[anchor=south]{$q\DD$}(5.5,0.5);
	\draw (5.5, 2)[fill=gray!5]	rectangle node[anchor=center] {$\rho\NN$}(7, 3);
	\draw (5.5, 1)[fill=gray!30]	rectangle node[anchor=center] {$\rho\ND$}(7, 2);
	\draw (5.5, 0)[fill=gray!30]	rectangle node[anchor=center] {$\rho\DD$}(7, 1);
\end{tikzpicture}
\caption{Notation for the various measures in the diffusion-decay equation. The measures $q_{\mu\nu}$ are pair (coupled) measures, with first and second marginals indicated to the left and right of the arrows. The various marginals $\init\rho_{\mu\nu}$ and $\rho_{\mu\nu}$ combine as indicated to form the observed normal ($\init\rho_N$ and $\rho_N$) and dark matter ($\init\rho_D$ and $\rho_D$) at the initial and final times.
 \label{fig: split measures}}
\end{figure}
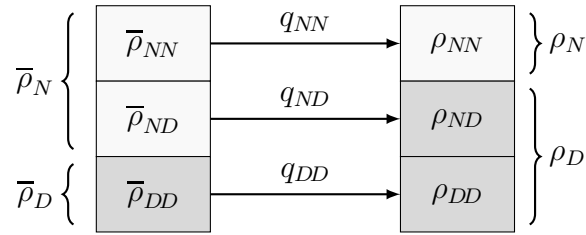
Theorem~\ref{theo: diff dec gamma convergence} below shows that for small $h$ we have $\J^h\DDec \approx \K^h\DDec$, where
\begin{equation}
\label{def:diff dec gradient flow}
\begin{split}
	\K^h\DDec(\rho\NN,\rho\ND,\rho\DD| \init\rho_N,\init\rho_D)\; := & - \tfrac12 \S(\rho\NN+\rho\ND) -\tfrac12 \S(\init\rho_N) + \tfrac{1}{4h}d^2(\init\rho_N,\rho\NN+\rho\ND)\\
																																	& + \tfrac12\S(\rho\DD) -\tfrac12\S(\init\rho_D) + \tfrac{1}{4h}d^2(\init\rho_D,\rho\DD) \\
																																	& + \S(\rho\NN) + \S(\rho\ND) - |\rho\NN| \log r^h\NN  - |\rho\ND|\log r^h\ND.
\end{split}
\end{equation}
 Let the admissible sets be:
\begin{align*}
	B^0										&:= \left\{(\init\rho_N,\init\rho_D) \in \mathcal{M}^+(\R^d)^2 : \init\rho_N+ \init\rho_D \in \P_2^\S(\R^d)\right\}; \\
	B(\init\rho_N,\init\rho_D) 	&:= \big\{(\rho\NN,\rho\ND,\rho\DD) \in \mathcal{M}^+(\R^d)^3 : \tfrac{1}{|\init\rho_N|}(\rho\NN+\rho\ND) \in \P_2^\S(\R^d) \\
	&\hspace{7cm} \text{ and } \tfrac{1}{|\init\rho_D|} \rho\DD \in \P_2^\S(\R^d)\big\},
\end{align*}
equipped with the product of the weak or strong topologies from Section~\ref{sec:Mosco convergence}. We remark that $(\rho\NN,\rho\ND,\rho\DD) \in B(\init\rho_N,\init\rho_D)$ implies that $|\init\rho_N| = |\rho\NN+\rho\ND| \text{ and } |\init\rho_D|=|\rho\DD|$.
\begin{theo}
\label{theo: diff dec gamma convergence}
  Assume that Conjecture~\ref{conj} holds. Then for all $(\init\rho_N,\init\rho_D) \in B^0$
	\begin{equation}
	\label{eq: diffusion splitting gamma convergence}
	\begin{split}
		&\J^h\DDec(\cdot\NN,\cdot\ND,\cdot\DD|\init\rho_N,\init\rho_D) - \frac{1}{4h}d^2(\init\rho_N\,,\cdot\NN+\cdot\ND) - \frac{1}{4h}d^2(\init\rho_D\,,\cdot\DD) \\
		&\hspace{8cm}+ |\cdot\ND\!|\,\log r^h\ND   + |\cdot\NN\!|\,\log r^h\NN \\
		&\qquad \xrightarrow[h\to 0]{M} \quad {}-\tfrac12 \S(\cdot\NN+\cdot\ND) - \tfrac12 \S(\init\rho_N) + \tfrac12 \S(\cdot\DD) - \tfrac12 \S(\init\rho_D) + \S(\cdot\NN) + \S(\cdot\ND) . \\
	\end{split}
	\end{equation}
	in $B(\init\rho_N,\init\rho_D)$.
\end{theo}
Note that we have not only subtracted three singular terms from $\J^h\DDec$, analogously to Theorem~\ref{theo: Fokker-Planck Gamma convergence}, but also the $h$-order term $-|\cdot\NN\!|\log r^h\NN$; the latter is for reasons of symmetry and to simplify calculations.


\bigskip

Finally, we show that the functional $\K^h\DDec$ in~\eqref{def:diff dec gradient flow} indeed defines a variational formulation of the diffusion-decay equation \eqref{eq: diffusion decay}. In view of completeness, and of generalisations to diffusion-reaction equations that we will discuss in Section~\ref{sec:extensions}, we prove convergence of the full scheme, including the dark matter, to the system of equations~\eqref{def:diff dec system}. We then derive the corresponding result for the single diffusion-decay equation~\eqref{eq: diffusion decay} by minimising over the dark matter (see Remark~\ref{contr:ignore-dark-matter} below), a procedure essentially the same as the contraction principle (Section~\ref{sec: large deviations}). Because we keep track of the dark matter, the matter that decays in a time step should be added to the dark matter already present from the previous iteration.


\begin{theo}
\label{theo: diff dec gradient flow} Let $\rho^0 \in \P^a_2(\R^d)$ and define the sequence $\{(\rho_N^{h,k}, \rho_D^{h,k})\}_{k \geq 0}$ by:
    \begin{subequations}
	\label{eq: diff dec approximation scheme}
	\begin{align}
		&(\rho_N^{h,0}, \rho_D^{h,0}) = (\rho^0,0), \notag \\
	\intertext{and for $k \geq 1$:} 
		&(\rho\NN^{h,k}, \rho\ND^{h,k}, \rho\DD^{h,k}) \in \underset{\rho\NN+\rho\ND+\rho\DD \in \P^a_2(\R^d)}{\arg \min} \,  	\K^h\DDec(\rho\NN,\rho\ND,\rho\DD|\rho_N^{h,k-1},\rho_D^{h,k-1}),  \\
		&(\rho_N^{h,k},\rho_D^{h,k})  = (\rho\NN^{h,k}, \rho\ND^{h,k} + \rho\DD^{h,k}).
		\label{eq: diff dec approximation scheme-2}
	\end{align}
	\end{subequations}
These minimisers exist uniquely, and as $h\to 0$ the pair $(\rho_N^{h,\lfloor t/h \rfloor}, \rho_D^{h,\lfloor t/h \rfloor})$ converges weakly in $L^1(\R^d\times(0,T))\times L^1(\R^d\times(0,T))$ to the solution of \eqref{def:diff dec system} with initial condition $(\rho^0,0)$.   
\end{theo}
The proof of this theorem is based on \cite{JKO1998}, and can  easily be extended to an additional drift term (see Section~\ref{sec:extensions}). Note that when we let $\lambda\to 0$ then $|\rho\ND|$ should vanish in \eqref{def:diff dec gradient flow} to prevent blow-up; indeed, in that case 
\begin{equation*}
  \K^h\DDec(\rho\NN,0,\rho\DD|\rho_N^{k-1},\rho_D^{k-1}) = \K^h\Diff(\rho\NN|\rho_N^{k-1}) + \K^h\Diff(\rho\DD|\rho_D^{k-1}).
\end{equation*}



\medskip

\begin{rem} 
\label{contr:ignore-dark-matter}
		A further contraction can be used to ignore the dark matter. We can then ignore the initial dark matter as well, so that the sequence $\tfrac1n \sum_{i=1\;:\;\nu^h_i=N}^n \delta_{Y^h_i}$ satisfies a large-deviation principle with rate $n$ and rate functional
					\begin{equation*}
						\rho_N \mapsto \inf_{\substack{0\leq \init\rho\NN\leq\init\rho_N \\ |\init\rho\NN|=|\rho_N|}}\; \inf_{q \in \Gamma(\init\rho\NN,\rho_N)} \mathcal{H}(q\NN|\init\rho\NN r\NN^h\theta^h).
					\end{equation*}
The corresponding energy-dissipation functional is then:
\begin{multline}
\label{eq: diff dec gradient flow normal matter}
\Kcontracted(\rho_N|\init\rho_N) :=  \inf_{\rho\ND : |\rho_N+\rho\ND|=|\init\rho_N|} 
	  -\tfrac12 \S(\rho_N+\rho\ND) -\tfrac12 \S(\init\rho_N) 
	  + \tfrac{1}{4h} d^2(\init\rho_N,\rho_N+\rho\ND)\\ 
+ \S(\rho_N) + \S(\rho\ND) -|\rho_N|\,\log r\NN^h - |\rho\ND|\, \log r\ND^h,
\end{multline}
which matches the minimisation problem~\eqref{min:intro-rhok-DD}.
The corresponding version of Theorem~\ref{theo: diff dec gradient flow} is
\begin{theo}
\label{theo:GF-no-dark-matter}
Let  $\rho^0 \in \P^a_2(\R^d)$ and define the sequence $\{\rho_N^{h,k}\}_{k \geq 0}$ by $\rho_N^{h,0} = \rho^0$ and for $k \geq 1$
\[
\rho_N^{h,k} \in \argmin_{\rho\in \M^+(\R^d)} \Kcontracted(\rho|\rho^{h,k-1}_N).
\]
These minimisers exist uniquely, and as $h\to 0$ the function $\rho_N^{h,\lfloor t/h \rfloor}$ converges weakly in $L^1(\R^d\times(0,T))$ to the solution of \eqref{eq: diffusion decay} with initial condition $\rho^0$.
\end{theo}
\end{rem}

\begin{rem} If we restrict ourselves to measures of mass $|\rho_N| = r^h\NN |\init\rho_N|$, thereby excluding the possible fluctuation in the decay process, then \eqref{eq: diff dec gradient flow normal matter} further reduces to
					\begin{equation*}
						\rho_N \mapsto \tfrac12 \S\!\left(\tfrac{1}{r\NN^h} \rho_N\right) - \tfrac12 \S(\init\rho_N) + \tfrac{1}{4h} d^2\!\left(\init\rho_N,\tfrac{1}{r\NN^h} \rho_N\right).
					\end{equation*}
          A similar scheme to deal with decaying mass can be found in \cite{KinderlehrerWalkington1999}.
\end{rem}

\subsection{Proof of Theorem~\ref{theo: diff dec gamma convergence}}


To reduce clutter we abbreviate $\rho\NT:=\rho\NN+\rho\ND$ and $q\NT:=q\NN+q\ND$. The sum over $\mu\nu=N\!N,N\!D$ in $\J^h\DDec$ can be rewritten as:
\begin{equation}
\label{eq: diffusion decay rate functional reformulation}
\begin{split}
	&\inf_{\init\rho\NN + \init\rho\ND = \init\rho_N} \sum_{\nu=N,D} \inf_{q_{N\nu}\in\Gamma(\init\rho_{N\nu}\!,\rho_{N\nu})}\mathcal{H}\!\left(q_{N\nu}\vert \init\rho_N r^h_{N\nu} \theta^h\right) \\
	&\quad = \inf_{\init\rho\NN + \init\rho\ND = \init\rho_N} \sum_\nu \inf_{q_{N\nu}\in\Gamma(\init\rho_{N\nu}\!,\rho_{N\nu})}\iint\!\log\!\left(    \frac{dq\NT}{d\init\rho_N\theta^h} \cdot     \frac{d\rho_{N\nu}}{d\rho\NT}    \cdot      \frac1{r^h_{N\nu}}     \cdot      \frac{dq_{N\nu}}{\frac{d\rho_{N\nu}}{d\rho\NT} dq\NT}\right)q_{N\nu} \\ 
	&\quad = \inf_{q\NT \in \Gamma(\init\rho_N,\rho\NT)} \mathcal{H}\!\left(q\NT \vert \init\rho_N\theta^h\right) + \S(\rho\NN) + \S(\rho\ND) - \S(\rho\NT) \\
	&\qquad - |\rho\NN|\,\log r^h\NN - | \rho\ND|\,\log r^h\ND + \inf_{\init\rho\NN + \init\rho\ND = \init\rho_N} \inf_{\substack{ q\NN+q\ND = q\NT \\ q\NN \in \Gamma(\init\rho\NN,\rho\NN)}} \sum_\nu \mathcal{H}\!\left(q_{N\nu} \Big\vert \frac{d\rho_{N\nu}}{d\rho\NT} q\NT\right).
\end{split}
\end{equation}
We now show that the last sum vanishes under the infima. Since $|q_{N\nu}| = |\rho_{N\nu}| = |\tfrac{d\rho_{N\nu}}{d\rho\NT} q\NT|$, we can apply Gibbs' inequality for $\nu=N,D$:
\begin{equation*}
		\mathcal{H}\!\left(q_{N\nu} \Big\vert \frac{d\rho_{N\nu}}{d\rho\NT} q\NT\right) \geq 0.
\end{equation*}
On the other hand, for any given $q\NT$, the measures
\begin{align*}
	\tilde q\NN := \frac{d\rho\NN}{d\rho\NT} q\NT, 		&& \tilde q\ND := \frac{d\rho\ND}{d\rho\NT} q\NT
\end{align*}
and their first marginals $\init\rho\NN(\cdot) = \tilde q\NN (\cdot \times \R^d)$ and $\init\rho\ND(\cdot) = \tilde q\ND (\cdot \times \R^d)$ are admissible in the infima. It follows that
\begin{equation}
\begin{split}
	\inf_{\init\rho\NN+\init\rho\ND=\init\rho_N} \inf_{\substack{ q\NN+q\ND = q\NT \\ q\NN \in \Gamma(\init\rho\NN,\rho\NN)}} \sum_\nu \mathcal{H}\!\left(q_{N\nu} \Big\vert \frac{d\rho_{N\nu}}{d\rho\NT} q\NT\right) 
																				&\leq \sum_\nu \mathcal{H}\!\left(\tilde q_{N\nu} \Big\vert \frac{d\rho_{N\nu}}{d\rho\NT} q\NT\right)=0.
\end{split}
\end{equation}
Hence we can write:
\begin{multline}
\label{eq:JhDec-rewrite}
	\J^h\DDec(\rho\NN,\rho\ND,\rho\DD|\init\rho_N,\init\rho_D) = \inf_{q\NT \in \Gamma(\init\rho_N,\rho\NT)} \mathcal{H}\!\left(q\NT \vert \init\rho_N\theta^h\right) +  \mathcal{H}\!\left(q\DD\vert \init\rho_D \theta^h\right) \\ + \S(\rho\NN) + \S(\rho\ND) - \S(\rho\NT) - |\rho\NN|\,\log r^h\NN - | \rho\ND|\,\log r^h\ND.
\end{multline}

Fix a $(\init\rho_N,\init\rho_D) \in B^0$. We first prove the lower bound of the Mosco convergence, and then the existence of a recovery sequence.

\emph{Lower Bound}. Take any narrowly convergent sequence
\begin{equation*}
  (\rho^h\NN,\rho^h\ND,\rho^h\DD) \weakto (\rho\NN,\rho\ND,\rho\DD) \qquad \text{in } B(\init\rho_N,\init\rho_D).
\end{equation*}
Again, we write $\rho^h\NT=\rho^h\NN+\rho^h\ND$. Combining~\eqref{def:diff dec rate functional}, \eqref{eq: diffusion splitting gamma convergence}, and~\eqref{eq:JhDec-rewrite}, we need to prove that:
\begin{equation}
\label{eq: diff dec gamma convergence lb}
\begin{split}
	&\liminf_{h\to 0} \inf_{q\NT \in \Gamma(\init\rho_N,\rho^h\NT)} \mathcal{H}\!\left(q\NT \vert \init\rho_N\theta^h\right) - \frac{1}{4h}d^2(\init\rho_N,\rho^h\NT)\\
	&\qquad + \inf_{q\DD \in \Gamma(\init\rho_D,\rho^h\DD)} \mathcal{H}(q\DD|\init\rho_D \theta^h) - \tfrac{1}{4h}d^2(\init\rho_D,\rho^h\DD) + \S(\rho^h\NN) + \S(\rho^h\ND) - \S(\rho^h\NT)\\
	&\quad\geq -\tfrac12 \S(\rho\NT) - \tfrac12 \S(\init\rho_N) + \tfrac12 \S(\rho\DD) - \tfrac12 \S(\init\rho_D)+ \S(\rho\NN) + \S(\rho\ND) .
\end{split}
\end{equation}
We will prove the lower bound for a number of terms separately.
\begin{itemize}
\item By assumption, $|\init\rho_N|^{-1} \rho\NT$ lies in $\P_2^\S(\R^d)$. If Conjecture~\ref{conj} is true for probability measures, it also holds for measures of different mass, so that:
	\begin{equation}
	\label{eq: diff dec gamma convergence lb ADPZ in N}
		\liminf_{h\to 0} \inf_{q\NT \in \Gamma(\init\rho_N,\rho^h\NT)} \mathcal{H}\!\left(q\NT \vert \init\rho_N\theta^h\right) - \frac{1}{4h}d^2(\init\rho_N,\rho^h\NT) \geq \tfrac12 \S(\rho\NT)-\tfrac12 \S(\init\rho_N).
	\end{equation}
	Similarly, $|\init\rho_D|^{-1}\rho\DD \in \P_2^\S(\R^d)$ and so:
	\begin{equation}
	\label{eq: diff dec gamma convergence lb ADPZ in D}
		\liminf_{h\to 0} \inf_{q\DD \in \Gamma(\init\rho_D,\rho^h\DD)} \mathcal{H}(q\DD|\init\rho_D \theta^h) - \tfrac{1}{4h}d^2(\init\rho_D,\rho^h\DD) \geq \tfrac12 \S(\rho\DD) - \tfrac12 \S(\init\rho_D).
	\end{equation}
\item Since the function $(x,y) \mapsto x\log x + y\log y - (x+y)\log(x+y)$ is convex, the functional
	\begin{equation*}
		F:(\rho\NN,\rho\ND) \mapsto \S(\rho\NN) + \S(\rho\ND) - \S(\rho\NN+\rho\ND)
	\end{equation*}
is also convex, and lower semicontinuous in $B(\init\rho_N,\init\rho_D)$ with the narrow topology \cite[Th. 4.3]{Giusti2003}
	\begin{equation}
	\label{eq: diff dec gamma convergence lb wlsc}
		\liminf_{h\to 0} \S(\rho^h\NN) + \S(\rho^h\ND) - \S(\rho^h\NT) \geq \S(\rho\NN) + \S(\rho\ND) - \S(\rho\NT).
	\end{equation}
\end{itemize}
The required lower bound \eqref{eq: diff dec gamma convergence lb} then follows from \eqref{eq: diff dec gamma convergence lb ADPZ in N}, \eqref{eq: diff dec gamma convergence lb ADPZ in D} and \eqref{eq: diff dec gamma convergence lb wlsc}.

\medskip

\emph{Recovery Sequence}. Fix $(\rho\NN, \rho\ND,\rho\DD) \in B(\init\rho_N,\init\rho_D)$ and take two recovery sequences $\rho^h\DD \to \rho\DD$ and $\rho^h\NT \to \rho\NN + \rho\ND$ in the strong topology from Conjecture~\ref{conj} such that
\begin{align}
\label{eq: diff dec gamma convergence rs ADPZ DD}
	&\limsup_{h\to 0} \inf_{q\DD \in \Gamma(\init\rho_D,\rho^h\DD)} \mathcal{H}(q\DD|\init\rho_D \theta^h) - \frac{1}{4h}d^2(\init\rho_D,\rho^h\DD) = \tfrac12 \S(\rho\DD) -\tfrac12 \S(\init\rho_D),\\
\label{eq: diff dec gamma convergence rs ADPZ N}
	&\limsup_{h\to0} \inf_{q\NT \in \Gamma(\init\rho_N,\rho^h\NT)} \mathcal{H}\!\left(q\NT \vert \init\rho_N\theta^h\right) - \frac{1}{4h}d^2(\init\rho_N,\rho^h\NT) = \tfrac12 \S(\rho\NN + \rho\ND) - \tfrac12 \S(\init\rho_N).
\end{align}
Contrary to the case of the lower bound we define $\rho^h\NN$ and $\rho^h\ND$ in terms of $\rho^h\NT$:
\begin{align*}
	\rho\NN^h := \frac{d\rho\NN}{d(\rho\NN+\rho\ND)}\rho^h\NT && \rho^h\ND := \frac{d\rho\ND}{d(\rho\NN+\rho\ND)}\rho^h\NT.
\end{align*}
Here we define the Radon-Nikodym derivatives to be $1$ on null sets of $\rho\NN+\rho\ND$. Observe that by definition of the strong topology $\S(\rho^h\NT)\to S(\rho\NN+\rho\ND)$. By Lemma~\ref{theo: strong entropy convergence}, this implies that $\rho^h\NT\to\rho\NN+\rho\ND$ and $\rho^h\NT\log\rho^h\NT\to(\rho\NN+\rho\ND)\log(\rho\NN+\rho\ND)$ strongly in $L^1({\R^d})$, if we redefine the sequence by its convergent subsequence. Therefore, with $0\leq\alpha(x):=\frac{d\rho\NN}{d(\rho\NN+\rho\ND)}(x)\leq 1$
\begin{equation*}
\begin{split}
  |\S(\rho^h\NN)-\S(\rho\NN)|&=\left|\int\!\alpha\rho^h\NT\log\alpha\rho^h\NT-\int\!\alpha\cdot(\rho\NN+\rho\ND)\log\alpha\cdot(\rho\NN+\rho\ND)\right|\\
    &\leq\left| \int\!\alpha\rho^h\NT\log\rho^h\NT-\int\!\alpha(\rho\NN+\rho\ND)\log(\rho\NN+\rho\ND)\right|\\
    &\qquad+\left|\int\!\rho^h\NT\alpha\log\alpha-\int\!(\rho\NN+\rho\ND)\alpha\log\alpha\right|\\
    &\leq \int\!\left|\rho^h\NT\log\rho^h\NT-(\rho\NN+\rho\ND)\log(\rho\NN+\rho\ND)\right|\\
    &\qquad + \frac1e\int\!\left|\rho^h\NT-(\rho\NN-\rho\ND)\right|\\
    &\to0,
\end{split}
\end{equation*}
and analogously for $\rho^h\ND$. Collecting the convergence results:
\begin{align}
\label{eq: diff dec gamma convergence rs entropies}
	&\S(\rho^h\NN) \to \S(\rho\NN), & \S(\rho^h\ND) \to \S(\rho\ND) && \text{and } && \S(\rho^h\NT) \to \S(\rho\NN+\rho\ND).
\end{align}
Then it follows from \eqref{eq: diff dec gamma convergence rs ADPZ DD}, \eqref{eq: diff dec gamma convergence rs ADPZ N}, and \eqref{eq: diff dec gamma convergence rs entropies} that $(\rho^h\NN,\rho^h\ND,\rho^h\DD)$ is a recovery sequence, ie.
\begin{equation*}
\begin{split}
	&\limsup_{h\to 0} \inf_{q\NT \in \Gamma(\init\rho_N,\rho^h\NT)} \mathcal{H}\!\left(q\NT \vert \init\rho_N\theta^h\right) - \frac{1}{4h}d^2(\init\rho_N,\rho^h\NT)\\
	&\qquad + \inf_{q\DD \in \Gamma(\init\rho_D,\rho^h\DD)} \mathcal{H}(q\DD|\init\rho_D \theta^h) - \frac{1}{4h}d^2(\init\rho_D,\rho^h\DD) + \S(\rho^h\NN) + \S(\rho^h\ND) - \S(\rho^h\NT) \\
	&\leq -\tfrac12 \S(\rho\NN+\rho\ND) - \tfrac12 \S(\init\rho_N) + \tfrac12 \S(\rho\DD) - \tfrac12 \S(\init\rho_D)+ \S(\rho\NN) + \S(\rho\ND) .
\end{split}
\end{equation*}
This concludes the proof of Theorem~\ref{theo: diff dec gamma convergence}.

\subsection{Proof of Theorem~\ref{theo: diff dec gradient flow}}

Theorem~\ref{theo: diff dec gradient flow} contains two main results: existence and uniqueness of minimisers, and the convergence of time-discrete solutions. We first discuss the existence and uniqueness of minimisers. By slightly rewriting~\eqref{eq: diff dec approximation scheme} we can minimise, for fixed $(\rho_N^{h,k-1},\rho_D^{h,k-1}) \in \P_2^a(\R^d)$, the functional
\begin{align}
	(\rho\NN,\rho\NT,\rho\DD) &\mapsto \K^h\DDec(\rho\NN,\rho\NT-\rho\NN,\rho\DD| \rho_N^{h,k-1},\rho_D^{h,k-1}) \notag\\
=&{} -\tfrac12 \S(\rho\NT) -\tfrac12 \S(\rho_N^{h,k-1}) + \tfrac{1}{4h}d^2(\rho_N^{h,k-1},\rho\NT)\notag\\
&{} + \tfrac12 \S(\rho\DD) -\tfrac12 \S(\rho_D^{h,k-1}) + \tfrac{1}{4h}d^2(\rho_D^{h,k-1},\rho\DD) \notag\\
& + \S(\rho\NN) + \S(\rho\NT-\rho\NN) - |\rho\NN|\log r\NN^h - |\rho\NT-\rho\NN|\log r\ND^h.
\label{eq: KDec to minimise}
\end{align}
The negative sign of the term $-\tfrac12 \S(\rho\NT)$ makes this minimisation problem slightly non-trivial. We therefore proceed in steps. 
For fixed $\rho\NT$, the functional 
\begin{equation*}
	F^h(\rho\NN) := \S(\rho\NN)+\S(\rho\NT-\rho\NN) - |\rho\NN|\,\log r\NN^h - |\rho\NT-\rho\NN|\,\log r\ND^h
\end{equation*}
is convex and has a unique stationary point that satisfies
\begin{equation*}
	0 = \log \rho\NN - \log (\rho\NT-\rho\NN) - \log r\NN^h + \log r\ND^h,
\end{equation*}
implying that $\rho\NN := r^h\NN \rho\NT$ is the unique global minimiser of $F$. Therefore, at every step $k$, we have (see Figure~\ref{fig: split measures})
\begin{equation}
\label{eq:values-of-rhonn}
\rho_N^{h,k} = \rho\NN^{h,k} = r^h\NN \rho\NT^{h,k} \quad\text{and}\quad
\quad\rho\ND^{h,k} = r^h\ND \rho\NT^{h,k}.
\end{equation}

The problem of minimising \eqref{eq: KDec to minimise} can now be reduced to the minimisation of
\begin{align}
	(\rho\NT,\rho\DD) \mapsto &\K^h\DDec(r\NN^h \rho\NT,r\ND^h\rho\NT,\rho\DD| \rho_N^{h,k-1},\rho_D^{h,k-1}) \notag\\
														&= \tfrac12 \S(\rho\NT) - \tfrac12 \S(\rho_N^{h,k-1}) + \tfrac{1}{4h}d^2(\rho_N^{h,k-1},\rho\NT)\notag\\
														&\quad + \tfrac12 \S(\rho\DD) -\tfrac12 \S(\rho_D^{h,k-1}) + \tfrac{1}{4h}d^2(\rho_D^{h,k-1},\rho\DD),
\label{eq: diff dec simplified scheme}
\end{align}
which consists of two decoupled minimisation problems, for which existence and uniqueness of minimisers are proved in \cite[Prop.~4.1]{JKO1998}. 

\bigskip

The compactness of the sequence $(\rho_N^{h,\lfloor t/h\rfloor},\rho_D^{h,\lfloor t/h\rfloor})$ is based on the same principle as in~\cite{JKO1998}, but with a twist. The central observation is again that $(\rho_N^{h,k-1},\rho_D^{h,k-1})$ is admissible in~\eqref{eq: diff dec simplified scheme}, leading to the estimate
\begin{equation}
\label{est:single-step}
\tfrac1{2h} d^2(\rho_N^{h,k-1},\rho\NT^{h,k}) + \tfrac{1}{2h}d^2(\rho_D^{h,k-1},\rho\DD^{h,k}) \leq 
- \S(\rho\NT^{h,k}) + \S(\rho_N^{h,k-1}) - \S(\rho\DD^{h,k}) + \S(\rho_D^{h,k-1}).
\end{equation}
However, the migration of mass from normal to dark matter means that upon summing this estimate over $k$, terms in the right-hand side do not cancel. 
Below we establish the \emph{a priori} estimates
\begin{align}
\label{est:aprioriM2}
&M_2(\rho_N^{h,k}+\rho_D^{h,k}) := \int\! |x|^2 d(\rho_N^{h,k}+ \rho_D^{h,k}) \leq C\\
&\sum_{k=1}^{\lfloor T/h\rfloor} d^2(\rho_N^{h,k-1},\rho\NT^{h,k}) + d^2(\rho_D^{h,k-1},\rho\DD^{h,k}) \leq Ch,
\label{est:telescope}
\end{align}
where the constant $C$ only depends on the initial data and on the maximal time~$T$.
As in~\cite{JKO1998} these provide the appropriate tightness in space (by~\eqref{est:aprioriM2}) and continuity in time (by~\eqref{est:telescope}) to conclude that there exists a subsequence such that $(\rho_N^{h,\lfloor t\slash h\rfloor},\rho_D^{h,\lfloor t\slash h\rfloor}) \to (u_N,u_D)$, weakly in $L^1(\R^d\times(0,T))\times L^1(\R^d\times(0,T))$.

We now prove~\eqref{est:aprioriM2} and~\eqref{est:telescope}. Recall from~\cite{JKO1998} the estimates
\begin{alignat}2
\label{eq: entropy lower bound}
	 -\S(\rho)  &\leq C\left( M_2(\rho) + 1 \right)^\alpha	&\qquad&\text{for some }0<\alpha<1\text{ and for all } \rho \in \mathcal{M}^+(\R^d),\\
	 M_2(\rho_1) &\leq 2 M_2(\rho_0) + 2 d^2(\rho_0,\rho_1) 																&\qquad&\text{for all } \rho_0,\rho_1 \in \mathcal{M}^+(\R^d) \text{ with } |\rho_0| = |\rho_1|.\notag
\end{alignat}
This allows us to estimate, for $n\in \N$ such that $nh\leq T$, 
\begin{equation}
\label{est:M2-applied}
M_2(\rho_N^{h,n}+\rho_D^{h,n}) \leq 2M_2(\rho_N^0+\rho_D^0) + 2d^2(\rho_N^{h,n}+\rho_D^{h,n},\rho_N^0+\rho_D^0).
\end{equation}
The second term above we then estimate by
\begin{eqnarray}
d^2(\rho_N^{h,n}+\rho_D^{h,n},\rho_N^0+\rho_D^0) 
&\leq& \biggl[\sum_{k=1}^n d(\rho_N^{h,k}+\rho_D^{h,k},\rho_N^{h,k-1}+\rho_D^{h,k-1})\biggr]^2\notag\\
&\leq& n \sum_{k=1}^n d^2(\rho_N^{h,k}+\rho_D^{h,k},\rho_N^{h,k-1}+\rho_D^{h,k-1})\notag\\
&=& n \sum_{k=1}^n d^2(\rho\NT^{h,k}+\rho\DD^{h,k},\rho_N^{h,k-1}+\rho_D^{h,k-1})\notag\\
&\stackrel{\eqref{ineq:Wasserstein-sum}
}\leq &n \sum_{k=1}^n d^2(\rho\NT^{h,k},\rho_N^{h,k-1}) + d^2(\rho\DD^{h,k},\rho_D^{h,k-1}).
\label{est:M2}
\end{eqnarray}
We also observe some properties of $\S$:
\begin{multline*}
\S(\alpha\rho+\beta\rho) = \S(\alpha\rho) + \S(\beta\rho) - \alpha|\rho|\log\frac \alpha{\alpha+\beta} - \beta|\rho|\frac\beta{\alpha+\beta}, \\
\text{for all $\alpha,\beta>0$ and $\rho\in \mathcal M^+(\R^d)$},
\end{multline*}
and in general
\begin{multline*}
\S(\rho_1 + \rho_2) \leq \S(\rho_1) + \S(\rho_2) - |\rho_1| \log \frac{ |\rho_1|}{|\rho_1+\rho_2|} - |\rho_2| \log \frac{ |\rho_2|}{|\rho_1+\rho_2|}\\
\text{for any $\rho_1,\rho_2 \in \mathcal{M}^+(\R^d)$.}
\end{multline*}
The first follows from simple calculation, and the second can be proved by writing $\rho_1+\rho_2 = \lambda (\rho_1/\lambda) + (1-\lambda) (\rho_2/(1-\lambda))$, applying the convexity of $\S$, and optimising with respect to~$\lambda$.
Combining these with~\eqref{eq:values-of-rhonn} we then have 
\begin{align}
&\S(\rho^{h,k}\NT) = \S(\rho^{h,k}\NN) +\S(\rho^{h,k}\ND)
  - |\rho^{h,k}\NN| \log r^h\NN - |\rho^{h,k}\ND| \log r^h\ND,\quad\text{and}
  \label{eq:split-rhont}\\
&\S(\rho^{h,k}_D) \leq \S(\rho^{h,k}\ND) + \S(\rho^{h,k}\DD) 
  - |\rho^{h,k}\ND | \log \frac{|\rho^{h,k}\ND|}{|\rho^{h,k}_D|}
  - |\rho^{h,k}\DD | \log \frac{|\rho^{h,k}\DD|}{|\rho^{h,k}_D|}.
  \label{est:split-rhod}
\end{align}

Now, putting the ingredients together:
\begin{eqnarray*}
&&\hskip-2cm \lefteqn{M_2(\rho_N^{h,n}+\rho_D^{h,n})\stackrel{\eqref{est:M2-applied}}\leq 2M_2(\rho_N^0+\rho_D^0) + 2d^2(\rho_N^{h,n}+\rho_D^{h,n},\rho_N^0+\rho_D^0)}\\
	&\stackrel{\eqref{est:M2}}\leq&
	  C + 2n\sum_{k=1}^{n} d^2(\rho_N^{h,k-1},\rho\NT^{h,k}) + d^2(\rho_D^{h,k-1},\rho\DD^{h,k}) \\
	&\stackrel{\eqref{est:single-step}}\leq& C+4nh\sum_{k=1}^n \S(\rho_N^{h,k-1})-\S(\rho\NT^{h,k})  + \S(\rho_D^{h,k-1}) - \S(\rho\DD^{h,k}) \\
	&\stackrel{\eqref{eq:split-rhont},\eqref{est:split-rhod}}\leq&C+ 4T\sum_{k=1}^n \S(\rho_N^{h,k-1}) - \S(\rho_N^{h,k}) + \S(\rho_D^{h,k-1}) - \S(\rho_D^{h,k}) \\
	&&\quad {}+ 4T\underbrace{\sum_{k=1}^n |\rho\NN^{h,k}|\log r\NN^h + |\rho^{h,k}\ND| \log r\ND^h  - |\rho\ND^{h,k}| \log \frac{|\rho\ND^{h,k}|}{|\rho_D^{h,k}|} - |\rho\DD^{h,k}| \log \frac{|\rho\DD^{h,k}|}{|\rho_D^{h,k}|}}_{{}\leq 0 \text{ (see below)}} \\
	&\stackrel{\eqref{eq: entropy lower bound}}\leq& C+ 4T\Bigl[\S(\rho_N^0) + \S(\rho_D^0) +C(M_2(\rho_N^{h,n})+1)^\alpha + C(M_2(\rho_D^{h,n})+1)^\alpha\Bigr] \\
	&\leq& C+ 4T\Bigl[\S(\rho_N^0) + \S(\rho_D^0)+ 2^\alpha C(M_2(\rho_N^{h,n}+\rho_D^{h,n})+2)^\alpha\Bigr].\end{eqnarray*}
Therefore $M_2(\rho_N^{h,n}+\rho_D^{h,n})$ is bounded on finite time intervals, which proves~\eqref{est:aprioriM2}, and 
the boundedness of the second line above implies~\eqref{est:telescope}. 

The sign of the brace above can be shown as follows: setting $r:=r\NN^h$ and therefore by~\eqref{eq:values-of-rhonn}, we have
\[
|\rho_N^{h,k}| = r^k, \quad |\rho_D^{h,k}| = 1-r^k,\quad
|\rho\ND^{h,k}| = r^k-r^{k-1},\quad \text{and}\quad |\rho\DD^{h,k}| = 1-r^{k-1}.
\]
Then
\begin{align*}
	&\sum_{k=1}^n |\rho\NN^{h,k}|\log r\NN^h + |\rho^{h,k}\ND| \log r\ND^h - |\rho\ND^{h,k}| \log \frac{|\rho\ND^{h,k}|}{|\rho_D^{h,k}|} - |\rho\DD^{h,k}| \log \frac{|\rho\DD^{h,k}|}{|\rho_D^{h,k}|} \notag\\
	&\quad = \sum_{k=1}^n r^k \log r + (r^{k-1}-r^k) \log (1-r)  - (r^{k-1}-r^k) \log \frac{r^{k-1}-r^k}{1-r^k} \notag\\
	&\hspace{10cm} - (1-r^{k-1}) \log \frac{1-r^{k-1}}{1-r^k}\notag \\
	&\quad = \sum_{k=1}^n r^k \log r^k - r^{k-1} \log r^{k-1} + (1-r^k) \log (1-r^k) - (1-r^{k-1}) \log (1-r^{k-1}) \notag\\
	&\quad = {} r^n \log r^n + (1-r^n) \log (1-r^n) \leq 0. 
\end{align*}
This concludes the proof of the compactness and therefore the convergence of a subsequence. 
\bigskip

We now determine the equation satisfied by the time-discrete minimisers using the method introduced in  \cite{JKO1998}. After perturbing the minimisers $\rho\NT^{h,k}$ and $\rho\DD^{h,k}$ by a push-forward, we find that for all $\xi \in C^\infty_0(\R^d;\R^d)$,
\begin{align}
\label{eq:after push forward perturbation}
	\iint\!(y-x)\cdot\xi(y)\,q\NT(dx\,dy) - h \!\int\!\div\xi(y)\,\rho\NT^{h,k}(y)\, dy =0,\notag\\
	\iint\!(y-x)\cdot\xi(y)\,q\DD(dx\,dy) - h \!\int\!\div\xi(y)\,\rho\DD^{h,k}(y)\, dy =0,
\end{align}
where $q\NT$ and $q\DD$ are the optimal transport plans in $d(\rho_N^{h,k-1},\rho\NT^{h,k})$ and $d(\rho_D^{h,k-1},\rho\DD^{h,k})$. Using $\rho_N^{h,k}=\rho\NN^{h,k}=r\NN^h \rho\NT^{h,k}$ and $\rho_D^{h,k}=r\ND^h \rho\NT^{h,k} + \rho\DD^{h,k}$ as prescribed by~\eqref{eq: diff dec approximation scheme-2} and~\eqref{eq:values-of-rhonn}, we add up the equations above to find for all $\xi$,
\begin{align}
\label{eq: diff dec summed euler-lagrange}
	\iint\!(y-x)\cdot\xi(y)\, r\NN^h q\NT(dx\,dy) - h\! \int\! \div\xi(y)\,\rho_N^{h,k}(y)\,dy =0, \notag \\
	\iint\!(y-x)\cdot\xi(y)\,(r\ND^h q\NT + q\DD)(dx\,dy) - h\! \int\!\div\xi(y)\,\rho_D^{h,k}(y)\,dy =0.
\end{align}
As $r\NN^h q\NT \in \Gamma(r\NN^h \rho_N^{h,k-1}, \rho_N^{h,k})$ and $r\ND^h q\NT + q\DD \in \Gamma(r\ND^h\rho_N^{h,k-1} + \rho_D^{h,k-1},\rho_D^{h,k})$, (although the second may not be optimal) we have the following bounds for any $\zeta \in C_0^\infty(\R^d)$:
\begin{equation*}
\begin{split}
	&\left\vert \int\!\left(\rho_N^{h,k}-r\NN^h \rho_N^{h,k-1}\right)\zeta - \iint(y-x)\cdot\grad\zeta(y)\, r\NN^h q\NT(dx\,dy)\right\vert\\
	&\quad=\left\vert\iint\!\left( \zeta(y)-\zeta(x) + (x-y)\cdot\grad \zeta(y)\right) r\NN^h q\NT(dx\,dy)\right\vert \\
	&\quad\leq\tfrac12 \sup\vert\lapl \zeta\vert \;r\NN^h \iint\!\left\vert y-x\right\vert^2 q\NT(dx\,dy)\\
	&\quad = \tfrac12 \sup\vert\lapl \zeta\vert \, d^2(\rho_N^{h,k-1},\rho\NT^{h,k}),
\end{split}
\end{equation*}
and similarly,
\begin{equation*}
	\begin{split}
		&\left\vert \int\!\left(\rho_D^{h,k}-(r\ND^h \rho_N^{h,k-1}+\rho_D^{h,k-1})\right) \zeta - \iint\!(y-x)\cdot\grad\zeta(y)\, (r\ND^h q\NT+q\DD)(dx\,dy)\right\vert \\
		&\quad \leq \tfrac12 \sup \vert\lapl \zeta\vert\, \left(d^2(\rho_N^{h,k-1},\rho\NT^{h,k}) + d^2(\rho_D^{h,k-1},\rho\DD^{h,k})\right).
	\end{split}
\end{equation*}
After applying these bounds to the equations \eqref{eq: diff dec summed euler-lagrange}, taking $\xi=\grad\zeta$, we find for all $\zeta$:
\begin{align*}
&\left\vert \int\!\left(\tfrac1h (\rho_N^{h,k} - r\NN^h\rho_N^{h,k-1})\,\zeta - \rho_N^{h,k} \lapl \zeta\right)dy \right\vert \leq \tfrac{1}{2h} \sup |\lapl \zeta|\, d^2(\rho_N^{h,k-1},\rho\NT^{h,k}),\\
\intertext{and}
&\left\vert \int\!\left(\tfrac1h (\rho_D^{h,k} - r\ND^h\rho_N^{h,k-1} - \rho_D^{h,k-1})\,\zeta - \rho_D^{h,k}\,\lapl \zeta \right)dy \right\vert \\
&\hspace{5cm}  \leq \tfrac{1}{2h} \sup |\lapl \zeta| \left(d^2(\rho_N^{h,k-1},\rho\NT^{h,k}) + d^2(\rho_D^{h,k-1},\rho\DD^{h,k})\right).
\end{align*}
Using the convergence of a subsequence (not relabeled) $(\rho_N^{h,\lfloor t\slash h\rfloor},\rho_D^{h,\lfloor t\slash h\rfloor}) \to (u_N,u_D)$ weakly in $L^1(\R^d\times(0,T))\times L^1(\R^d\times(0,T))$, we find that for all $\zeta \in C^\infty_0(\R^d\times[0,T])$,
\begin{equation*}
\begin{split}
  &\left\vert \int_0^ T \!\int\! u_N\left(-\partial_t \zeta + \left(\lim_{h\to 0}\tfrac{1-r\NN^h}{h}\right) \zeta - \lapl\zeta\right) dy\,dt \right\vert\\
  &\qquad \xleftarrow{h\to 0} \left\vert \int_0^ T\!\int\!\left( \tfrac1h \left(\rho_N^{h,\lfloor t/h\rfloor}-\rho_N^{h,\lfloor t/h\rfloor-1}\right) \zeta + \tfrac{1-r\NN^h}h \rho_N^{h,\lfloor t/h\rfloor -1}\,\zeta-\rho_N^{h,\lfloor t/h\rfloor}\,\lapl \zeta \right) dx\,dt\right\vert\\
  &\qquad \leq \sum_{k=1}^{\lfloor T/h\rfloor} \tfrac12 \sup \left|\lapl {\textstyle \int_0^ T}\!\zeta\right| \,d^2(\rho_N^{h,k-1},\rho\NT^{h,k})\\
  &\qquad \stackrel{\eqref{est:telescope}}\leq C h \xrightarrow{h\to 0} 0,
\end{split}
\end{equation*}
and for the dark matter:
\begin{equation*}
\begin{split}
  &\left\vert \int_0^T \!\int\!\left(  -u_D\,\partial_t\zeta - \left(\lim_{h\to 0}\tfrac{r\ND^ h}{h}\right)  u_N\, \zeta - u_D\,\lapl\zeta  \right) dy\,dt \right\vert\\
  &\qquad \xleftarrow{h\to 0} \left\vert \int_0^ T\!\int\!\left( \tfrac1h \left(\rho_D^{h,\lfloor t/h\rfloor}-\rho_D^{h,\lfloor t/h\rfloor-1}\right) \zeta - \tfrac{r\ND^h}h \rho_N^{h,\lfloor t/h\rfloor -1}\,\zeta-\rho_D^{h,\lfloor t/h\rfloor}\, \lapl \zeta \right) dx\,dt\right\vert\\
  &\qquad \leq \sum_{k=1}^{\lfloor T/h\rfloor} \tfrac12 \sup \left|\lapl {\textstyle \int_0^ T}\!\zeta\right| \,\left(  d^2(\rho_N^{h,k-1},\rho\NT^{h,k}) + d^2(\rho_D^{h,k-1},\rho\DD^{h,k})\right)\\
  &\qquad \stackrel{\eqref{est:telescope}}\leq C h \xrightarrow{h\to 0} 0.
\end{split}
\end{equation*}
From this we see that the limit $(u_N,u_D)$ indeed solves \eqref{def:diff dec system} (weakly in $L^1(\R^d\times(0,T))$). 
This concludes the proof of Theorem~\ref{theo: diff dec gradient flow}.

\subsection{Drift with decay and reactions}
\label{sec:extensions}

\paragraph{Diffusion with drift and decay.} The results from Sections~\ref{sec: FP} and \ref{sec: diff dec} can be easily combined in the following way. A microscopic model for the Fokker-Planck equation with decay \eqref{eq: diffusion drift decay} is obtained by replacing the spatial transition probability $\theta^h$ in the micro model from Section~\ref{sec: diff dec micro model} by the fundamental solution $\eta^h$ of the Fokker-Planck equation from Definition~\ref{def: fundamental solution}. The corresponding large-deviation rate functional then simply becomes \eqref{def:diff dec rate functional} with that transition probability. By the same arguments of Theorems~\ref{theo: Fokker-Planck Gamma convergence} and \ref{theo: diff dec gamma convergence}, the large-deviation rate functional is related to the following energy-dissipation functional in a Mosco-convergence sense:
\begin{equation}
\label{def:KFPDec}
\begin{split}
\K^h\FPDec (\rho\NN,\rho\ND,\rho\DD| \init\rho_N,\init\rho_D) \;:=\;
  & - \tfrac12 \S(\rho\NN+\rho\ND) -\tfrac12 \S(\init\rho_N) + \tfrac{1}{4h}d^2(\init\rho_N,\rho\NN+\rho\ND) \\
  & + \tfrac12 \S(\rho\DD) -\tfrac12\S(\init\rho_D) + \tfrac{1}{4h}d^2(\init\rho_D,\rho\DD) \\
	& + \S(\rho\NN) + \S(\rho\ND) - |\rho\NN|\,\log r^h\NN - |\rho\ND|\,\log r^h\ND\\
	& + \tfrac12 \E(\rho\NN+\rho\ND+\rho\DD)- \tfrac12 \E(\init \rho_N+\init\rho_D).
\end{split}
\end{equation}

Indeed, as our main result this functional defines a variational formulation for the Fokker-Planck equation with decay \eqref{eq: diffusion drift decay}:
\begin{theo}
Let  $\rho^0 \in \P^a_2(\R^d)$ and define the sequence $\{(\rho_N^{h,k}, \rho_D^{h,k})\}_{k \geq 0}$ by:
  \begin{subequations}
	\begin{align*}
		&(\rho_N^{h,0}, \rho_D^{h,0}) = (\rho^0,0), \notag \\
	\intertext{and for $k \geq 1$:} 
		&(\rho\NN^{h,k}, \rho\ND^{h,k}, \rho\DD^{h,k}) \in \underset{\rho\NN+\rho\ND+\rho\DD \in \P^a_2(\R^d)}{\arg \min} \,  	\K^h\FPDec(\rho\NN,\rho\ND,\rho\DD|\rho_N^{h,k-1},\rho_D^{h,k-1}),  \\
		&(\rho_N^{h,k},\rho_D^{h,k})  = (\rho\NN^{h,k}, \rho\ND^{h,k} + \rho\DD^{h,k}).
	\end{align*}
	\end{subequations}
	These minimisers exist uniquely, and as $h\to 0$ the pair $(\rho_N^{h,\lfloor t/h \rfloor}, \rho_D^{h,\lfloor t/h \rfloor})$ converges weakly in $L^1(\R^d\times(0,T))$ to the solution of \eqref{def:diff dec system} with initial condition $(\rho^0,0)$.   
\end{theo}
The proof is a slight adaptation of the proof of Theorem~\ref{theo: diff dec gradient flow}, with the observation that after perturbing with a push-forward, the continuity equations \eqref{eq:after push forward perturbation} include the additional terms $h\!\int\!\xi(y)\cdot\grad\Psi(y)\,\rho\NT^{h,k}(y)\,dy$ and $h\!\int\!\xi(y)\cdot\grad\Psi(y)\,\rho\DD^{h,k}(y)\,dy$ for the potential energy. Following the proof of Theorem~\ref{theo: diff dec gradient flow}, these extra terms will result in the convection term in equation~\eqref{eq: diffusion drift decay}.

\bigskip

\paragraph{Diffusion-reaction equations.} 

Another useful generalisation is a system of equations that describe the transition between a set of states $\nu$ in some index set $I$:
	\begin{equation}
  \label{eq:diff reaction}
		\partial_t u_\nu = \lapl u_\nu - \sum_{\mu\neq\nu} s_{\mu\nu} u_\nu + \sum_{\mu\neq\nu} s_{\nu\mu} u_\mu,\hspace{1.5cm} \nu \in I.
	\end{equation}
	We should then choose the transition probabilities $r^h_{\mu\nu}$ of the microscopic system in such a way that $\lim_{h\to0} \frac{r^h_{\mu\nu}}{h}=s_{\mu\nu}$ and $r^h_{\mu\mu}=1-\sum_{\nu\neq\mu} r^h_{\mu\nu}$. The large-deviation rate functional corresponding to this micro model is:
	\begin{equation*}
		\left(\left\{\rho_{\mu\nu}\right\}_{\mu,\nu\in I};\left\{\init\rho_\mu\right\}_{\mu\in I}\right) \mapsto \sum_{\mu\in I} \inf_{\substack{\init\rho_{\mu\nu}:\mu,\nu\in I\\ \sum_{\nu\in I} \init\rho_{\mu\nu}=\init\rho_\mu}} \sum_{\nu\in I} \inf_{q_{\mu\nu}\in\Gamma(\init\rho_{\mu\nu},\rho_{\mu\nu})} \mathcal{H}\!\left(q_{\mu\nu}\vert \init\rho_\mu \, r_{\mu\nu}^h \theta^h\right)\!,
	\end{equation*}
	which Mosco-converges, after subtracting singular terms, to the functional:
	\begin{equation}
	\label{def: diff reaction gradient flow}
\sum_{\mu\in I}	\Big\lbrack - \tfrac12 S\big({\textstyle\sum_{\nu\in I}}\, \rho_{\mu\nu}\big) - \tfrac12 \S(\init\rho_\mu) + \tfrac{1}{4h} d^2\big(\init\rho_\mu,{\textstyle \sum_{\nu\in I}}\, \rho_{\mu\nu}\big) + \sum_{\nu\in I} \left( \S(\rho_{\mu\nu}) - |\rho_{\mu\nu}|\,\log r_{\mu\nu}^h \right) \Big\rbrack.
	\end{equation}	
In the same way as in Theorem~\ref{theo: diff dec gradient flow}, this functional defines a variational formulation for the system of diffusion-reaction equations \eqref{eq:diff reaction}. 

\section{Discussion}
\label{sec: discussion}

The work of~\cite{Adams2011} uncovered an intriguing link between the diffusion equation, the entropy-Wasserstein gradient-flow formulation of that equation, and a large-deviation principle for a stochastic particle system. The work of the present paper is motivated by the question whether this  link can be generalised. 

Equation~\eqref{eq: diffusion drift decay} moves beyond~\cite{Adams2011} in two ways. The additional drift term represented by $\Psi$ is  compatible with the Wasserstein framework. The corresponding equation~\eqref{eq: Fokker-Planck} is a Wasserstein gradient flow of the free energy functional $\S+\E$. In Section~\ref{sec: FP} we show that also the large-deviation connection generalises to this case, with only minor modification. Corresponding continuous-time large-deviations results for instance in~\cite{DawsonGartner94} or~\cite[Th.~13.37]{FengKurtz06} mirror this.

The case of decay is different. The structure of the time-discretised gradient flow in Theorem~\ref{theo: diff dec gradient flow}  has some non-standard features:
\begin{itemize}
\item The iteration defined in Theorem~\ref{theo: diff dec gradient flow} is special in that the minimisation is taken over the pair $(\rho\NN,\rho\ND)$, and the result is \emph{added} to the dark matter of the previous time step. Of course, when ignoring the dark matter, as in Remark~\ref{contr:ignore-dark-matter}, this is not visible, as is shown in the corresponding definition in Theorem~\ref{theo:GF-no-dark-matter}.

\item The functional $\K^h\DDec$ in~\eqref{def:diff dec gradient flow} is not that of a `standard' gradient flow. The discussion in Section~\ref{sec: Wasserstein gradient flows} and the proof of Theorem~\ref{theo: diff dec gradient flow} suggests to split it into three parts; two parts that represent the diffusion steps for normal and decayed matter, and a third part for the decay step. The fact that the operator can be split into terms for each driving force is related to the indepence of the processes in the micro model, so that the transition probability is a product of two probabilities, which can then be split according to calculation \eqref{eq: diffusion decay rate functional reformulation}. Pursuing the analogy with the diffusion step, and with metric-space gradient flows, one might interpret $\S(\rho\NN)+\S(\rho\NT-\rho\NN)-\S(\rho\NT)$ as the as the driving energy behind the decay, by which the dissipation would then become the (linear!) terms $-|\rho\NN|\log r^h\NN  - |\rho\ND|\log r^h\ND$. In which sense this interpretation is meaningful is as yet unknown.

\end{itemize}



 
The way we have set up the microscopic model in this paper restricts us to decay processes. The reason that we cannot generalise to `birth' processes (i.e. $\lambda<0$) is that, in the microscopic model, linear birth rates depend on the amount of existing normal matter. Therefore, in contrast to exponential decay, exponential birth requires a system of particles with interdependence, which prevents the techniques in this paper to be extended to birth processes in a trivial way.

\medskip

The exact choice of the microscopic transition probabilities may not influence the continuum limit, as the limit only depends on asymptotic behaviour of the probabilities $r_{\mu\nu}^h$ as $h\to0$. However, this choice will affect the discrete-time approximation~\eqref{def:diff dec gradient flow}. In general, different microscopic systems can lead to different variational formulations for the same equation. For instance, the minimisation functional~\eqref{def: diff reaction gradient flow} that we derive for a system of diffusion-reaction equations differs from the $L^2$-gradient flow in~\cite{Peletier2010} for that same equation, as the underlying microscopic model of that paper models reaction as diffusion in a chemical landscape.

\medskip

One of the interesting suggestions of the connection between large-deviation principles and gradient flows is the possibility that \emph{every} gradient-flow structure might correspond to a large-deviation principle for \emph{some} stochastic process. For instance, there is of course a different gradient-flow formulation for the diffusion-decay equation without drift~\eqref{eq: diffusion decay}, with driving energy
\[
\mathcal \E(\rho) := \int\Bigl[\frac12 |\grad \rho|^2 + \frac\lambda 2 \rho^2\Bigr]\, dx,
\]
and with the $L^2$-metric as dissipation. This can be seen by using the fact that in the Hilbert space $L^2$  a gradient flow satisfies at each time $t>0$
\[
(\partial _t \rho,s)_{L^2} = -\langle\E'(\rho),s\rangle\qquad\text{for all }s\in L^2,
\]
which can be rewritten as a weak form of~\eqref{eq: diffusion decay}. Could this structure be related to a large-deviation principle of some stochastic process? At this point we have no idea.

\appendix
\section{The quenched large-deviation principle}

In this appendix we derive the large-deviation principles that are used in this paper - in a slightly more general context. First we state the large-deviation principle of the pair empirical measure. The proof is mainly due to L\'eonard, but we include it here to provide the full details. In the following, $\Omega$ will denote a (separable metric) Radon space. 
\begin{theo}[{\cite[Prop.~3.2]{Leonard2007}}]
\label{theo: pair measure LDP}
Fix $\rho^0 \in \P(\Omega)$ and let $\left\{x_i\right\}_{i=1,\hdots,n,n\geq 1} \subset \Omega$ be so that
\begin{equation}
\label{eq: general quenched initial condition}
	L_n^0 := \frac{1}{n} \sum_{i=1}^n \delta_{x_i} \longweakto \rho^0 \quad \text{ as } n\to \infty.	
\end{equation}
Let $\zeta: \Omega \to \P(\Omega)$ be continuous with respect to the narrow topology of $\P(\Omega)$, and let each random variable $Y_i$ in $\Omega$ be distributed by $\zeta^{x_i}$. Define the pair empirical measure $M_n := n^{-1} \sum_{i=1}^n \delta_{(x_i,Y_i)}$. Then the sequence $\{M_n\}_n$ satisfies the large-deviation principle in $\P(\Omega^2)$ with rate $n$ and rate functional:
	\begin{equation}
	\label{def: pair rate functional}
		I(q)         :=	\begin{cases}	\mathcal{H}(q|p),	&\text{if } q(\cdot\,\times\Omega) = \rho^0(\cdot), \\
																				\infty,									&\text{otherwise},
										\end{cases}
	\end{equation}
	with $p(dx\,dy):= \zeta^x(dy)\rho^0(dx)$. 
\end{theo}

\begin{proof} We write $C_b(\Omega^2)$ for the space of continuous bounded functions on $\Omega^2$, and $C_b(\Omega^2)^\ast$ and $C_b(\Omega^2)'$ for its topological and algebraic dual respectively, the latter being the space of all linear functionals on $C_b(\Omega^2)$ with the weakest topology that makes all these linear functionals continuous. We equip both $C_b(\Omega^2)^\ast$ and $C_b(\Omega^2)'$ with the topology induced by the duality with $C_b(\Omega^2)$, denoted by $\langle\,\cdot,\cdot\,\rangle$. Recall that the dual $C_b(\Omega)^\ast$ can be identified with the space of finite, finitely additive, and regular signed Borel measures \cite[Th. IV.6.2]{DunfordSchwartz1957}. Moreover, since $\Omega^2$ is Radon any probability measure is regular. Hence $\P(\Omega^2) \subset C_b(\Omega^2)^\ast \subset C_b(\Omega^2)'$, and the topologies on $\P(\Omega^2)$ and $C_b(\Omega^2)^\ast$ coincide with the induced topology as a subset of $C_b(\Omega^2)'$. Note, however, that $C_b(\Omega^2)^\ast$ is closed, while $\P(\Omega^2)$ is not.


We first consider $M_n$ as random variables in $C_b(\Omega^2)'$. For an arbitrary $d\in\N$ and $\phi_1,\hdots,\phi_d$ in $C_b(\Omega^2)$, define the new random variables:
\begin{equation*}
\begin{split}
	Z_{\phi_1,\hdots,\phi_d;n}	&:= \left(\langle \phi_1, M_n\rangle, \hdots, \langle \phi_d, M_n \rangle\right) \\
															& = \left(\tfrac{1}{n}\sum_{i=1}^n \langle \phi_1, \delta_{(x_i, Y_i)} \rangle, \hdots, \tfrac{1}{n}\sum_{i=1}^n \langle \phi_d, \delta_{(x_i, Y_i)} \rangle\right) \\
															& = \left(\tfrac{1}{n}\sum_{i=1}^n \phi_1(x_i, Y_i ), \hdots, \tfrac{1}{n}\sum_{i=1}^n \phi_d(x_i , Y_i )\right)\!.
\end{split}
\end{equation*}
First we prove the large-deviation principle of $\Law(Z_{\phi_1,\hdots,\phi_d;n})$ in $\R^d$, using the G\"artner-Ellis Theorem. For any $\lambda \in \R^d$:
\begin{equation}
\label{eq: Lambda phi n}
\begin{split}
	\Lambda_{\phi_1,\hdots,\phi_d;n}(\lambda) &:= \tfrac{1}{n} \log\left(\mathbb{E} \exp\!\left(n \lambda \cdot Z_{\phi_1,\hdots,\phi_d;n} \right) \right)\\
																						& =  \tfrac{1}{n} \log\left(\mathbb{E} \exp\!\left(\sum_{j=1}^d \sum_{i=1}^n \lambda_j  \phi_j(x_i ,Y_i ) \!\right)\!\right)\\
																						& \stackrel{(*)}{=} \tfrac{1}{n} \log\left(\prod_{i=1}^n \mathbb{E} \, \exp\!\left( \sum_{j=1}^d \lambda_j  \phi_j(x_i ,Y_i )\!\right)\!\right)\\
																						& =  \tfrac{1}{n} \sum_{i=1}^n \log \!\left( \int\! \exp\!\left( \sum_{j=1}^d \lambda_j\phi_j(x_i ,y)\!\right) \zeta^{x_i }(dy)\!\right) \\
																						& =  \int\! \tfrac{1}{n} \sum_{i=1}^n  \log \!\left( \int\! \exp\!\left( \sum_{j=1}^d \lambda_j\phi_j(x,y)\!\right) \zeta^x(dy)\!\right)\delta_{x_i }(dx) \\
																						& = \int \log \left(\int \! \exp\!\left(\sum_{j=1}^d\lambda_j{\phi_j}(x,y)\right) \zeta^x(dy)\!\right) L_n^0(dx) \\
																						& = \int \log \langle e^{\lambda\cdot\phi^x}\!, \, \zeta^x \rangle L_n^0(dx),  
\end{split}
\end{equation}
using the notation $\phi^x: y \mapsto (\phi_1(x,y),\hdots,\phi_d(x,y))$. In $(*)$ we have used the independence of $(x_i ,Y_i )$ to take the sum out of the expectation. 

In order to use \eqref{eq: general quenched initial condition} to pass to the limit $n\to \infty$ in \eqref{eq: Lambda phi n}, we need to show that $x\mapsto\log \langle e^{\lambda\cdot\phi^x}\!, \, \zeta^x\rangle$ is a  bounded and continuous function. The boundedness follows directly from the fact that all $\phi_j$ are bounded. To prove continuity, take any convergent sequence $x_m\to x$. As $\zeta^x$ is continuous as a function from $x\in \Omega$ to $\P(\Omega)$, Prokhorov's Theorem gives tightness of the sequence $\zeta^{x_m}$. Hence for each $\epsilon>0$ there exists a compact set $K_\epsilon \subseteq \Omega$ such that:
\begin{equation*}
	\zeta^{x_m}(\Omega\backslash K_\epsilon) < \epsilon \text{ for all } m\geq 1.
\end{equation*}
Using that the sequence of functions $y\mapsto e^{\lambda\cdot\phi^{x_m}(y)}$ converges uniformly on compact sets as $m\to\infty$, we have:
\begin{align*}
		&| \langle e^{\lambda\cdot\phi^{x_m}}\!, \zeta^{x_m} \rangle - \langle e^{\lambda\cdot\phi^x}\!,\zeta^x\rangle | = |\langle e^{\lambda\cdot\phi^{x_m}}-e^{\lambda\cdot\phi^x}\!, \zeta^{x_m} \rangle + \langle e^{\lambda\cdot\phi^x}\!, \zeta^x-\zeta^{x_m} \rangle | \\
		&\quad\leq  \int_{\Omega\backslash K_\epsilon}\!\! \big| e^{\lambda\cdot\phi^{x_m}(y)}-e^{\lambda\cdot\phi^x(y)}\big|\, \zeta^{x_m}(dy) \\
		&\qquad + \int_{K_\epsilon}\!\! \big| e^{\lambda\cdot\phi^{x_m}(y)}-e^{\lambda\cdot\phi^x(y)}\big|\,  \zeta^{x_m}(dy) + \big| \langle e^{\lambda\cdot\phi^x}\!, \zeta^x-\zeta^{x_m} \rangle \big|\\
		&\quad\leq (\|e^{\lambda\cdot\phi^{x_m}}\|_{L^\infty(\Omega)} + \|e^{\lambda\cdot\phi^x}\|_{L^\infty(\Omega)}) \,\underbrace{\zeta^{x_m}(\Omega\backslash K_\epsilon)}_{<\epsilon} \\
		&\qquad + \underbrace{\| e^{\lambda\cdot\phi^{x_m}}- e^{\lambda\cdot\phi^x} \|_{L^\infty(K_\epsilon)}}_{\to 0} \, \zeta^{x_m}(K_\epsilon) + \underbrace{\big| \langle e^{\lambda\cdot\phi^x}\!, \zeta^x-\zeta^{x_m} \rangle \big|}_{\to 0} \\
		&\quad \xrightarrow{m\to \infty} 2\epsilon \|e^{\lambda\cdot\phi^{x}}\|_{L^{\infty}(\Omega)}
\end{align*}
for arbitrary small $\epsilon>0$. Hence indeed $\langle e^{\lambda\cdot\phi^x}\!, \zeta^x\rangle$ is continuous in $x$, so we can apply \eqref{eq: general quenched initial condition} to find the limit:
\begin{equation*}
	\Lambda_{\phi_1,\hdots,\phi_d}(\lambda) := \lim_{n \to \infty} \Lambda_{\phi_1,\hdots,\phi_d;n}(\lambda) = \int \log \langle e^{\lambda\cdot\phi^x}\!, \zeta^x \rangle \rho^0(dx).
\end{equation*}
Since this function is continuously differentiable and finite throughout its whole domain ($\R^d$), the conditions of the G\"artner-Ellis Theorem \cite[Th.~2.3.6c]{Dembo1998} are met, so that $Z_{\phi_1,\hdots,\phi_d;n}$ satisfies the large-deviation principle in $\R^d$ with rate $n$ and rate function $\Lambda^\ast_{\phi_1,\hdots,\phi_d}$, the Fenchel-Legendre transform of $\Lambda_{\phi_1,\hdots,\phi_d}$.

Next we apply the Dawson-G\"artner Theorem \cite[Th.~4.6.9]{Dembo1998} to find that the sequence $\{M_n\}_n$ satisfies the large-deviation principle in $C_b(\Omega^2)'$ with rate $n$ and rate functional:
\begin{equation*}
\begin{split}
	I(q)	&:= \sup_{d \geq 1} \, \sup_{\phi_1,\hdots \phi_d \in C_b(\Omega^2)} \, \Lambda^\ast_{\phi_1, \hdots, \phi_d}\left(\left(\langle \phi_1, q\rangle, \hdots, \langle\phi_d, q \rangle \right)\right) \\
				& = \sup_{d \geq 1} \, \sup_{\phi_1,\hdots \phi_d \in C_b(\Omega^2)} \sup_{\lambda \in \R^d} \lambda \cdot \left(\langle \phi_1, q \rangle, \hdots, \langle \phi_d, q \rangle \right) - \Lambda_{\phi_1, \hdots, \phi_d}(\lambda) \\
				& = \sup_{\phi_\in C_b(\Omega^2)} \langle \phi, q \rangle - \int\!\log\langle e^{\phi^x},\zeta^x\rangle\rho^0(dx),
\end{split}
\end{equation*}
where as before we write $\phi^x:y\mapsto\phi(x,y)$. 

We now show that this rate functional is indeed \eqref{def: pair rate functional}. Since $C_b(\Omega^2)^*$ is a closed subset of $C_b(\Omega^2)'$ containing $\P(\Omega^2)$, we have $I=\infty$ on $C_b(\Omega^2)'\backslash C_b(\Omega^2)^\ast$ \cite[Th. 4.1.5]{Dembo1998}. Therefore, we only need to consider $q\in C_b^\ast(\Omega^2)$. 

\begin{itemize}
\item First, we show that $I(q)=\infty$ whenever $q\in C_b^*(\Omega^2)$ with first marginal $\pi^1 q \neq \rho^0$. This can be seen by restricting the supremum to $\phi$'s that depend on the first variable only:
\begin{equation*}
	\begin{split}
		I(q)	& \geq \sup_{\phi \in C_b(\Omega)} \langle \phi, q \rangle - \int \log \langle e^{\phi^x} , \zeta^x \rangle \rho^0(dx) \\
							&	= \sup_{\phi \in C_b(\Omega)} \langle \phi, \pi^1 q\rangle - \langle \phi, \rho^0 \rangle \\
							& =	\begin{cases}	0, 				&\text{if } \pi^1 q = \rho^0, \\
																		+\infty, 	&\text{otherwise}.
									\end{cases}
	\end{split}
\end{equation*}

\item Next, we show that $I(q)=\infty$ for any $q\in C_b(\Omega^2)^\ast$ that is finitely, but not countably additive. By the argument above, we only need to consider non-negative finitely additive measures with $q(\Omega^2) = 1$. For such $q$, there exists a sequence of disjoint measurable sets $A_i \subset \Omega^2$ such that 
\begin{equation*}
  \delta := q(\bigcup_{i=1}^\infty A_i) - \sum_{i=1}^\infty q(A_i) > 0.
\end{equation*}
Without loss of generality, assume that $\bigcup_{i=1}^\infty A_i=\Omega^2$. Since $q$ and $p$ are regular, one can find for any $k\geq 1$, sequences of sets $K_i \subset A_i \subset O_i$ with $K_i$ compact and $O_i$ open, such that:
\begin{align}
\label{eq:QLDP inner outer regular}
  \sum_{i=1}^\infty q(O_i)\leq 1 - \tfrac12 \delta &&\text{and}&& \sum_{i=1}^\infty p(A_i\backslash K_i) \leq e^{-k}.
\end{align}
Then for each $k,n\geq 1$ there exist a continuous function $\phi_{kn}: \Omega^2 \to \lbrack-k,0\rbrack$ such that
\begin{equation*}
	\phi_{kn}(x,y) = \begin{cases} -k,      & \text{on } \bigcup_{i=1}^n\!K_i,\\
                                 0,       & \text{on } \Omega^2\backslash \bigcup_{i=1}^n \!O_i.
                   \end{cases}
\end{equation*}
For these functions we have, on one hand (as $O_i$ might not be disjoint)
\begin{equation}
\label{eq:outer reg test}
  \langle\phi_{kn},q\rangle \geq -k\, q(\bigcup_{i=1}^n O_i) \geq -k \sum_{i=1}^n q(O_i),
\end{equation}
and on the other hand
\begin{equation*}
  \langle e^{\phi_{kn}^x}\!,\zeta^x\rangle \leq \int\!\left(e^{-k} \chi_{\bigcup_{i=1}^n K_i}(x,y)+\chi_{\Omega^2\backslash \bigcup_{i=1}^n K_i}(x,y)\right)\zeta^x(dy),
\end{equation*}
so that
\begin{equation}
\label{eq:QLDP inner reg test}
  \begin{split}
    \int\!\log\langle e^{\phi_{kn}^x},\zeta^x\rangle\rho^0(dx) 
    &\;\,\leq \int\!\left(-k + \log \int\!\left(\chi_{\bigcup_{i=1}^n K_i}+e^k \chi_{\Omega^2\backslash \bigcup_{i=1}^n K_i}\right)\zeta^x\right)\rho^0(dx)\\
    &\mathop{\leq}^\text{Jensen} -k + \log\left(p\left(\bigcup_{i=1}^n K_i\right) + e^k p\left(\Omega^2\backslash\bigcup_{i=1}^n K_i\right)\right).
  \end{split}
\end{equation}
Finally, we find for the rate functional:
\begin{eqnarray*}
		I(q) &\geq& \limsup_{k\to\infty} \; \limsup_{n\to\infty}\;  \langle\phi_{kn},q\rangle - \int\!\log\langle e^{\phi_{kn}^x},\zeta^x\rangle\rho^0(dx)\\
         &\stackrel{\eqref{eq:outer reg test},\eqref{eq:QLDP inner reg test}}\geq& \limsup_{k\to\infty} \; \limsup_{n\to\infty}\; -k \sum_{i=1}^n q(O_i) + k \\
         &&\hspace{4.6cm}-\log\left(p\!\left(\bigcup_{i=1}^n K_i\right) + e^k\,p\!\left(\Omega^2\backslash\bigcup_{i=1}^n K_i\right)\right)\\
         & = & \limsup_{k\to\infty} \; -k \sum_{i=1}^\infty q(O_i) + k - \log\left(p\!\left(\bigcup_{i=1}^\infty K_i\right) + e^k\,p\!\left(\Omega^2\backslash\bigcup_{i=1}^\infty K_i\right)\right)\\
         &\stackrel{\eqref{eq:QLDP inner outer regular}}\geq& \limsup_{k\to\infty} \; -k\,(1-\tfrac12\delta) + k - \log 2\\
         &=& \limsup_{k\to\infty} \; \tfrac12 \delta\, k - \log 2 = \infty.
\end{eqnarray*}

\item Now assume that $q\in\mathcal{P}(\Omega^2)$ such that $\pi^1 q = \rho^0$. The Disintegration Theorem then allows us to write
\begin{equation*}
	q(dx\,dy) = \rho^0(dx) q^x(dy)
\end{equation*}
for some family of measures $\{q^x: x \in \Omega\}$. In this case:
\begin{eqnarray*}
I(q) 	& = &\sup_{\phi \in C_b(\Omega^2)} \int \left(\langle \phi^x, q^x \rangle - \log \langle e^{\phi^x}, \zeta^x \rangle \right) \rho^0(dx) \\
		& \leq& \int \sup_{\phi^x \in C_b(\Omega)} \{ \langle \phi^x,q^x \rangle - \log \langle e^{\phi^x}, \zeta^x \rangle \} \rho^0(dx) \\
		& =& \int \mathcal{H}(q^x |\zeta^x) \rho^0(dx) \\
		& =& \begin{cases}
			\iint \!\left( \log \frac{d(\rho^0 q^x)}{d(\rho^0\zeta^x)}(x,y) \right) \rho^0(dx) q^x(dy),	&\text{if } \rho^0 q^x \ll \rho^0 \zeta^x,\\
			\infty,		&\text{otherwise}
		\end{cases}\\
		& =& \mathcal{H}(q|p).
\end{eqnarray*}

\item We conclude the proof with the inequality in the other direction. Observe that $I$ is the Fenchel-Legendre transform of
\begin{equation*}
\begin{split}
	\Lambda:\phi	&\mapsto \int\!\log\langle e^{\phi^x},\zeta^x\rangle\rho^0(dx) \\
								&\leq \log \int \langle e^{\phi^x} , \zeta^x \rangle \rho^0(dx) = \log\langle e^\phi, p\rangle,
\end{split}
\end{equation*}
where the bound follows from Jensen's inequality. Hence:
\begin{equation*}
	I(q) = \Lambda^\ast(q) \geq \sup_{\phi \in C(\Omega^2)} \{\langle \phi,q\rangle - \log\langle e^\phi,p\rangle \} = \mathcal{H}(q|p).
\end{equation*}
\end{itemize}
Since the large-deviation principle holds in $C_b(\Omega^2)^\ast$ with $D_I \subset \P(\Omega^2)$, it also holds in $\P(\Omega^2)$ with the same rate functional (i.e. restricted to $\P(\Omega^2)$) \cite[Th. 4.1.5]{Dembo1998}.
\end{proof}

The following corollary follows immediately from the contraction principle:
\begin{coro}
\label{coro: contracted quenched LDP}
	The sequence $\left\{n^{-1}\sum_{i=1}^n \delta_{Y_i }\right\}_n$ satisfies the large-deviation principle in $\P(\Omega)$ with rate $n$ and rate function:
	\begin{equation}
		J(\rho)         := \begin{cases}	\inf_{q \in \Gamma(\rho^0,\rho)} \mathcal{H}(q|p),	&\text{if } q \in \Gamma(\rho^0,\rho), \\
																			\infty,																												&\text{otherwise}.
											 \end{cases}
	\end{equation}
\end{coro}

\begin{rem}
\label{rem: quenched LDP}
	A straightforward approach would be to look for a large-deviation principle in the set of probability measures:
	\begin{equation}
		A \mapsto \mathbb{P}(M_n \in A |  L_n^0 = \rho^0).
	\end{equation}
	However, these conditional probabilities are not well-defined: the events $\{L_n^0 = \rho^0\}$ typically have zero probability. One way to deal with this is to condition on small neighbourhoods of $\rho^0$ of size $\delta$ instead, calculate the large-deviation rate functional for these conditional probabilities, and then take the limit for $\delta\to 0$. This is the approach taken in \cite{Adams2011}. We note that because the limits $n\to\infty$ and $\delta\to 0$ can not be interchanged, this approach does not a priori yield a large-deviation principle in the rigorous sense. 

In the approach that we adopt from \cite{Leonard2007}, we consider fixed initial positions so that there is no need to define the conditional probabilities above. This technique is sometimes called a \emph{quenched large-deviation principle}.
\end{rem}

\bibliographystyle{siam}
\bibliography{library}

\end{document}